\documentclass{article}
\usepackage{amsmath}
\usepackage{amsfonts}
\usepackage{amsthm}
\usepackage{mathrsfs}
\usepackage{amssymb}
\usepackage{tikz}
\usepackage{comment}
\usetikzlibrary{calc}
\usepackage{tikz-cd}
\usepackage{graphicx}
\usepackage[nodisplayskipstretch]{setspace}
\usepackage{varwidth}

\DeclareMathOperator{\Pic}{Pic}
\DeclareMathOperator{\Sym}{Sym}
\DeclareMathOperator{\Proj}{Proj}
\DeclareMathOperator{\Spec}{Spec}
\DeclareMathOperator{\Eff}{Eff}
\DeclareMathOperator{\vol}{vol}

\DeclareMathOperator{\sgn}{sgn}
\DeclareMathOperator{\conv}{conv}

\DeclareMathOperator{\DHeck}{DH}
\renewcommand{\P}{\mathbb P}
\renewcommand{\div}{\mathrm{div}}
\newcommand{\Oo}{\mathcal O}

\newcommand{\Q}{\mathbb Q}

\newcommand{\R}{\mathbb R}
\newcommand{\Z}{\mathbb Z}
\newcommand{\C}{\mathbb C}

\newtheorem{theorem}{Theorem}[section]
\newtheorem{lemma}[theorem]{Lemma}
\newtheorem{corollary}[theorem]{Corollary}
\newtheorem{proposition}[theorem]{Proposition}
\newtheorem{conjecture}[theorem]{Conjecture}

\theoremstyle{definition}
\newtheorem{definition}[theorem]{Definition}
\newtheorem{example}[theorem]{Example}

\theoremstyle{remark}
\newtheorem{remark}[theorem]{Remark}

\numberwithin{theorem}{section}

\title{Effective Divisors and Newton-Okounkov Bodies of Hilbert Schemes of Points on Toric Surfaces}
\author{Ian Cavey}

\begin{document}

\maketitle

\begin{abstract}
We compute the (unbounded) Newton-Okounkov body of the Hilbert scheme of points on $\C^2$. We obtain an upper bound for the Newton-Okounkov body of the Hilbert scheme of points on any smooth toric surface. We conjecture that this upper bound coincides with the exact Newton-Okounkov body for the Hilbert schemes of points on $\P^2, \P^1\times\P^1$, and Hirzebruch surfaces. These results imply upper bounds for the effective cones of these Hilbert schemes, which are also conjecturally sharp in the above cases. 
\end{abstract}

\section{Introduction}

For a smooth, algebraic surface $X$ over $\C$, the Hilbert scheme of $n$ points on $X$ parametrizes length $n$, zero-dimensional subschemes of $X$. A well-known theorem due to Fogarty states that this Hilbert scheme, denoted $X^{[n]}$, is a smooth, irreducible variety of dimension $2n$ \cite{Fo1}. In this paper we study effective divisors on the Hilbert schemes of points on toric surfaces using methods from the theory of Newton-Okounkov bodies.\\ 

Newton-Okounkov bodies are convex bodies associated to divisors on algebraic varieties, generalizing the connection between Newton polytopes and toric varieties. These convex bodies were introduced in passing by Okounkov \cite{O1}\cite{O2}, and their theory was further developed by Kaveh-Khovanskii \cite{KK} and Lazarsfeld-Musta\c{t}\u{a} \cite{LM}. We briefly recall the construction. Let $Y$ be a $d$-dimensional irreducible variety and $D$ a Cartier divisor on $Y$. The Newton-Okounkov body of $D$ depends on a choice of valuation $\nu:\C(Y)^\times\to \Z^d$ defined on the field of rational functions $\C(Y)$. After fixing $\nu$, the valuations obtained from sections of $\Oo(D)$ and its multiples can be assembled into a graded semigroup,
\[ \Gamma(D) = \bigoplus_{m\geq 0}\Gamma_m(D) =  \bigoplus_{m\geq 0} \{(m,\nu(f)) \hspace{1ex} | \hspace{1ex} f\in H^0(Y,\Oo(mD))^\times \}\subseteq \Z\times \Z^d, \]
and the Newton-Okounkov body $\Delta(D)$ is defined to be the closed convex hull,
\[ \Delta(D) = \overline{\conv\left( \bigcup_{m>0} \frac{1}{m}\Gamma_m(D) \right)}\subseteq \{1\}\times\R^d\simeq\R^d. \]

When $Y$ is projective and $D$ is a big divisor, the convex set $\Delta(D)\subseteq \R^d$ is bounded, and has the property that its Euclidean volume is equal to the volume of $D$ as a divisor (after normalizing by a factor of $d!$). It is sometimes more convenient to talk about the Newton-Okounkov body $\Delta(\mathscr{L})$ of a line bundle $\mathscr{L}$ on $Y$, which is defined by replacing $\Oo(mD)$ with $\mathscr{L}^{\otimes m}$ in the above construction. \\

We first study the Hilbert scheme of $n$ points on $\C^2$. In \cite{Ha2}, Haiman identifies $(\C^2)^{[n]}$ with an explicit blowup of the symmetric power $(\C^2)^{(n)}$. This identification equips $(\C^2)^{[n]}$ with an ample line bundle, which we denote by $\Oo(1)$. Our first result is a computation of the Newton-Okounkov body of this line bundle.

\begin{theorem}\label{C2intro} 
The Newton-Okounkov body of the line bundle $\Oo(1)$ on $(\C^2)^{[n]}$ is the closed convex hull of the set of $n$-tuples of distinct pairs $(a_1,b_1),\dots,(a_n,b_n)\in \Z_{\geq 0}^{2}$, labeled in increasing lexicographic order. This unbounded, convex polyhedron is defined by the inequalities
\[ \setstretch{1.5}
\Delta\left(\Oo(1)\right) = \left\{
\begin{varwidth}{10 in} \setstretch{1.5}
$(a_1,\dots,a_n,$\\
 $b_1,\dots,b_n)\in \R^{2n}$
 \end{varwidth} 
 \hspace{1ex}\left|\hspace{1ex} 
 \begin{varwidth}{10 in}
$0\leq a_1\leq a_2\leq\cdots\leq a_n$, and\\
$b_j\geq (j-i)(1-a_j)+a_i+\cdots +a_{j-1} $,\\
for all $1\leq i \leq j \leq n$
 \end{varwidth} 
 \right.
 \right\}.
 \setstretch{1}
\]
\end{theorem}

This theorem is proved in Section \ref{c2section}. Haiman's construction identifies the global sections of $\Oo(m)$ with certain polynomials, and the Newton-Okounkov body above is computed using a trailing term valuation on these polynomials (Definition \ref{val}). As $(\C^2)^{[n]}$ is not projective, the spaces of global sections of $\Oo(m)$ are infinite-dimensional, and accordingly the polyhedron $\Delta(\Oo(1))$ is unbounded. However, $\Delta(\Oo(1))$ still encodes asymptotic information about the sections of $\Oo(m)$ equivariantly, via the Duistermaat-Heckman measure (see Section \ref{DH}). These polyhedra have interesting combinatorial structure: for instance their top-dimensional bounded faces are enumerated by Catalan numbers. \\

Next we study the Hilbert schemes of points on smooth, projective, toric surfaces. We recall Fogarty's description of the Picard group $\Pic(X^{[n]})$ from \cite{Fo2}. Let $X$ be a smooth projective surface with irregularity $q(X)=0$ (which is the case whenever $X$ is toric). There is a linear embedding $\Pic(X)\subseteq \Pic(X^{[n]})$, which we denote $D\mapsto D_n$. Geometrically, if $D$ is the class of a smooth, irreducible curve $C\subseteq X$, then $D_n$ is represented by the locus of length $n$ subschemes of $X$ whose supports meet the curve $C$. The exceptional locus of the Hilbert-Chow morphism is an irreducible divisor on $X^{[n]}$ consisting of the nonreduced subschemes of $X$, whose class we denote by $B$. For notational convenience we often use the divisor class $E=-\frac12B$, which corresponds more directly to the line bundle $\Oo(1)$ on $(\C^2)^{[n]}$, instead of the geometrically defined divisor $B$. With this identification, there is an isomorphism $\Pic(X^{[n]})\simeq \Pic(X)\times \Z E$.\\

When $X$ is a smooth, projective, toric surface, $\Pic(X)$ is generated by torus invariant divisors. In Section \ref{toric} we recall the definition of the Newton polytope $P_D$ of such a divisor $D$ and identify $P_D$ with a subset of $\R^2$, writing
\[ P_D = \left\{ (a,b)\in \R^2 \hspace{1ex}\left|\hspace{1ex} 
 \begin{varwidth}{10 in} \setstretch{1.5}
$0\leq a \leq c$, and\\
$\ell(a) \leq b \leq u(a)$
 \end{varwidth} 
 \right.
 \right\}\]
 for some constant $c$, and piecewise linear functions $\ell$ and $u$.

\begin{theorem}\label{ubintro}
Let $D$ be a torus invariant divisor on a smooth, projective, toric surface $X$, with $P_D$ as above. For any $r\in \Z$, the Newton-Okounkov body $\Delta(D_n+rE)$ is contained in the convex set
\[ \setstretch{1.5}
\left\{ \begin{varwidth}{10 in} \setstretch{1.5}
$(a_1,\dots,a_n,$\\
\hspace{2ex} $b_1,\dots,b_n)\in \R^{2n}$
 \end{varwidth}  \hspace{1ex}\left|\hspace{1ex} 
 \begin{varwidth}{10 in} \setstretch{1.5}
$0\leq a_1\leq a_2\leq\cdots\leq a_n \leq c$, and\\
$b_j\geq \ell(a_j) +(j-i)(r-a_j) + a_i+\cdots +a_{j-1}$,\\
$b_j\leq u(a_j) - (k-j)(r+a_j) + a_{j+1}+\cdots+a_k$,\\
for all $1\leq i \leq j \leq k\leq n$
 \end{varwidth} 
 \right.
 \right\}.
 \setstretch{1}
\]
\end{theorem}

This is proved in Section \ref{projsection}. For $P_D = \R^2_{\geq 0}$, the convex set defined above recovers the Newton-Okounkov body of $\Oo(r)$ on $(\C^2)^{[n]}$. We refer to the convex body appearing in Theorem \ref{ubintro} as $\overline{\Delta}(D_n+rE)$, so the theorem asserts the inclusion
\[\Delta\left(D_n+rE\right) \subseteq \overline{\Delta}\left(D_n+rE\right). \]
For most toric surfaces $X$, this is a strict containment, but based on explicit computations for small $n$ we propose the following.

\begin{conjecture}\label{sharpintro}
If the surface $X$ is $\P^2, \P^1\times \P^1$, or a Hirzebruch surface, then the containment in Theorem \ref{ubintro} is sharp for all torus invariant divisors $D\in\Pic(X)$ and all $r\in \Z$. In other words, we conjecture the equality 
\[\Delta\left(D_n+rE\right) = \overline{\Delta}\left(D_n+rE\right) \]
for all divisors $D_n+rE\in \Pic(X^{[n]})$ in these cases.
\end{conjecture}

In Section \ref{examples} we verify this conjecture in the case of $n=4$ points on $\P^2$ to illustrate the methods leading to Conjecture \ref{sharpintro}, as well as the geometric information that these convex sets encode about the Hilbert schemes. \\

Of particular interest are the cones of effective divisors on $X^{[n]}$, which have been studied extensively \cite{ABCH},\cite{BC},\cite{Hu},\cite{Ry}. Huizenga \cite{Hu} has computed the effective cones of $(\P^2)^{[n]}$ for all $n$, but for other surfaces the effective cones are known only for small $n$. Theorem \ref{ubintro} implies an upper bound for the cone of effective divisors on $X^{[n]}$. Indeed, effective divisors $\xi\in\Pic(X^{[n]})$ have nonempty Newton-Okounkov bodies $\Delta(\xi)$, so if the upper bound $\overline{\Delta}(\xi)\supseteq \Delta(\xi)$ is empty, then $\xi$ is not effective. To compute the implied upper bound on the effective cones, it is convenient to use the global Newton-Okounkov body and its convexity properties, as explained in Section \ref{effsection}.\\

Conjecture \ref{sharpintro} would imply that this method computes the exact effective cones for the Hilbert schemes of points on $\P^2$, $\P^1\times\P^1$, and Hirzebruch surfaces. We have verified that this upper bound agrees with the effective cones of $(\P^2)^{[n]}$ computed by Huizenga for all $n\leq 171$ \textit{numerically} (see Section \ref{effsection}). Ryan \cite{Ry} has computed the effective cones for the Hilbert schemes of $n\leq 16$ points $\P^1\times\P^1$, and the upper bound is sharp in these cases as well. We have also computed similar bounds for Hirzebruch surfaces.\\

Upper bounds on the effective cone are often obtained by intersecting with moving curve classes. This differs from the approach described above, which comes instead from a valuation on the effective divisors. The valuation records order-of-vanishing information about individual effective divisors, whereas intersection products with curves depend only on the linear (indeed numerical) equivalence class of the divisor. Characterizing the set of valuations of all effective divisors in a given linear equivalence class is a large refinement of the problem of characterizing which classes contain an effective divisor. Given Conjecture \ref{sharpintro}, however, this finer invariant appears to yield simpler results, at least asymptotically. Indeed, the conjectural global Newton-Okounkov body is described by a list of explicit inequalities, uniform in the number of points $n$. Explicit descriptions of the effective cones however, which are projections of the global Newton-Okounkov bodies, appear to depend on the arithmetic properties of $n$ \cite{BC}. A table containing some data on effective cones computed using this method can be found at the end of Section \ref{effsection}.\\

\textit{Acknowledgements:} I thank the authors of \cite{ABCH} and \cite{Ha1}, whose exposition enabled me to learn about Hilbert schemes of surfaces. I am grateful to Izzet Coskun for helpful discussion during the early stages of this project. Most of all I thank David Anderson for teaching me about Newton-Okounkov bodies, providing detailed comments on the many iterations of this document, and suggesting the problem of computing Newton-Okounkov bodies of Hilbert schemes in the first place.\\

\textit{Conventions:} We use the term \textit{variety} to mean \textit{algebraic variety over} $\C$, and \textit{divisor} always means \textit{Cartier divisor}. We equip $\Z^d$ with the lexicographic order.\\

\section{Background}\label{bg}

\subsection{Newton-Okounkov Bodies}

We refer to \cite{KK} and \cite{LM} for proofs of the results stated in this section. Let $Y$ be a $d$-dimensional, irreducible variety over $\C$. A \textit{valuation} on $Y$ is a group homomorphism $\nu:\C(Y)^\times \to \Z^d$ such that
\begin{itemize}
\item $\nu(f+g)\geq \min\{\nu(f),\nu(g)\}$ for all $f,g\in \C(Y)^\times$ (with the lexicographic order on $\Z^d$, as always), and
\item $\nu(\lambda) = 0$ for all nonzero constant functions $\lambda\in \C^\times \subseteq\C(Y)^\times$.
\end{itemize}

One says that $\nu$ has \textit{one-dimensional leaves} if for any $a\in \Z^d$ the \textit{leaf at a}, 
\[ F_a= \{ f\in \C(Y)^\times | \nu(f)\geq a \} /  \{ f\in \C(Y)^\times | \nu(f)> a \}, \]
is a vector space of dimension at most one. This property implies that for any finite-dimensional linear subspace $V\subseteq \C(Y)$, the set $\{ \nu(f) | f\in V\setminus\{0\} \}\subseteq \Z^d$ has exactly $\dim(V)$ elements.\\

Fix a divisor $D$ on $Y$, and a valuation $\nu$ with one dimensional leaves. We consider $H^0(X,\Oo(D))\subseteq\C(Y)$ identified with the set of rational functions $f$ such that $D+\div(f)\geq 0$. The \textit{graded semigroup} of $D$ with respect to $\nu$ is defined as
\[ \Gamma_\nu(D) := \{ (\nu(f),k) \hspace{1ex}|\hspace{1ex} 0\neq f\in H^0(X,\Oo(kD)),  k\geq 0 \} \subseteq \Z^d\times \Z. \]
The grading here refers to the recording of which multiple $\Oo(kD)$ that each valuation comes from. One checks that $\Gamma_\nu(D)$ indeed forms a semigroup under the usual coordinate-wise addition of vectors.\\

We consider $\Gamma_\nu(D)\subseteq \Z^d\times \Z\subseteq \R^{d+1}$ in the obvious way. The \textit{cone} of $\Gamma_\nu(D)$, denoted $\Sigma_\nu(D)$, is the smallest closed, convex, cone containing the entire semigroup $\Gamma_\nu(D)\subseteq \R^{d+1}$. The \textit{Newton-Okounkov body} $\Delta_\nu(D)$ of $D$ is the intersection of the cone $\Sigma_\nu(D)$ with the affine subspace $\R^d\times\{1\}\subseteq \R^{d+1}$. We consider $\Delta_\nu(D)\subseteq \R^d$, and write $\Delta(D),\Gamma(D),$ and $\Sigma(D)$ for these objects when the choice of valuation is clear, or is unimportant.\\

In our computations, we identify the global sections of divisors $\Oo(D)$ with certain sets of polynomials (see Section \ref{algsection}). Importantly, these identifications are compatible with multiplication of global sections. In other words, if sections $s\in H^0(Y,\Oo(D))$ and $t\in H^0(Y,\Oo(E))$ are identified with polynomials $f$ and $g$ respectively, then the section $s\otimes t \in H^0(Y,\Oo(D)\otimes \Oo(E)) \simeq H^0(Y,\Oo(D+E)) $ is identified with $f\cdot g$. We use a leading/trailing term valuation on the associated polynomials, and this multiplication property ensures that the resulting sets of valuations still form a semigroup. \\

The volume of a divisor $D$ on a projective variety $Y$ is the asymptotic growth rate of sections of its multiples,
\[ \vol_Y(D):=\lim_{k\to\infty} \frac{h^0(Y,\Oo(kD))}{k^d/d!}. \]
A divisor $D$ is said to be \textit{big} if $\vol(D)>0$. In this case $\Delta(D)$ is bounded, and is therefore a convex body. The following fundamental result relates the Euclidean volume of the Newton-Okounkov body $\Delta(D)\subseteq \R^d$ to the volume of the divisor $D$ under these hypotheses.

\begin{theorem}\label{vol}[\cite{KK}\cite{LM}\cite{O2}]
Let $D$ be a big divisor on a projective variety $Y$. Then for any valuation $\nu$ on $Y$ with one-dimensional leaves, we have
\[ \vol_{\R^d}(\Delta_{\nu}(D)) = \frac{1}{d!}\vol_Y(D). \]
\end{theorem}

Under the same hypotheses, the convex body $\Delta(D)$ depends only on the numerical equivalence class of $D$. The Newton-Okounkov body of a big numerical class $\xi\in N^1(Y)$ is therefore defined to be $\Delta(D)$, where $D$ is any divisor in the class $\xi$. We also refer to the the Newton-Okounkov body $\Delta(\mathscr L)$ of a big line bundle $\mathscr L$, which is defined to be $\Delta(D)$ for any divisor $D$ such that $\mathscr L\simeq \Oo(D)$. Newton-Okounkov bodies are homogeneous, in the sense that $\Delta(kD) = k\Delta(D)$ for any integer multiple $k>0$. Using this homogeneity, Newton-Okounkov bodies can be defined for any big, rational class $\xi\in N^1(Y)_\Q$.\\

The global Newton-Okounkov body $\Delta(Y)\subseteq N^1(Y)_\R\times \R^d$ captures the notion that the convex bodies $\Delta(\xi)$ vary continuously in the argument $\xi\in N^1(Y)_\Q$. This set $\Delta(Y)$ is a closed, convex cone with the property that, under the projection to the first coordinate, the fiber of $\Delta(Y)\to N^1(Y)$ over any big rational class $\xi \in N^1(Y)_\Q$ is precisely $\Delta(\xi)\subseteq \R^d$. The global Newton-Okounkov body can be used to define $\Delta(\xi)$ for any real numerical class $\xi\in N^1(X)_\R$, by setting
\[ \Delta(\xi) = \pi^{-1}(\xi)\subseteq \R^d, \] 
where $\pi:\Delta(Y)\to N^1(Y)_\R$ is projection onto the first coordinate.  For $\xi$ not in the closure of the big cone (equivalently $\xi$ not pseudo-effective), $\Delta(\xi)$ is empty. For $\xi\in N^1(Y)$ on the boundary of the pseudo-effective cone, this convex set $\Delta(\xi)$ may differ from the convex sets constructed from the valuations of sections of representatives $D\in\xi$. In fact, different representatives $D,D'\in \xi$ may produce different convex bodies from the valuation construction when $\xi$ is on this boundary. To avoid this confusion, the fibers of the global Newton-Okounkov are called \textit{numerical Newton-Okounkov bodies} in \cite{B}, however we simply refer to them as \textit{Newton-Okounkov bodies}. 

\subsection{Toric Surfaces}\label{toric}

For the relevant background on toric surfaces we follow Section 6.1 of \cite{LM}, but restrict to the two-dimensional case. We use the notation and definitions established in \cite{Fu}.\\

A toric surfaces is constructed from a fan $\Sigma$ in $N_\R\simeq\R^2$ where $N\simeq \Z^2$ is a two-dimensional lattice. Each cone $\sigma$ in the fan corresponds to an affine variety $U_\sigma=\Spec(\C[\sigma^\vee\cap M])$, where $\sigma^\vee$ is the dual cone to $\sigma$, and $M$ is the dual lattice to $N$. These affine varieties are then glued together to form the toric surface $X=X_\Sigma$. In particular, the cone $\sigma = \{0\}$ gives an open set $T= \Spec(\C[M])\simeq (\C^*)^2$ inside of $X$, the two-dimensional algebraic torus. The action of $T$ on itself by coordinate-wise multiplication extends to an action of $T$ on the whole surface $X$. Lattice points $m\in M$ index rational functions $\chi^m$ on $X$.  \\

We assume that $X$ is smooth and projective, and both of these properties can be detected from the fan $\Sigma$. The surface $X$ is smooth if and only each two-dimensional cone $\sigma\in \Sigma$ is spanned by integral vectors $v,v'\in N$ such that $v$ and $v'$ generate the lattice $N$. Since $X$ is two-dimensional it is projective if and only if it is complete, and $X$ is complete if and only if the cones in $\Sigma$ cover the whole vector space $N_\R$.\\

Fix an ordering $v_1,\dots,v_s$ for the generators of the rays in $\Sigma$ so that the two-dimensional cones are spanned by consecutive rays $[v_1,v_2],\dots,[v_{s-1},v_s]$, and $[v_s,v_1]$. By the orbit-cone correspondence, the torus fixed points of $X$ are indexed by two dimensional cones $\sigma\in \Sigma$, and the $T$-invariant curves of $X$ are indexed by the rays of $\Sigma$. Let $D_i$ denote the $T$-invariant curve corresponding to the ray spanned by $v_i$. Let $D=\sum_{i=1}^s d_iD_i$ be a $T$-invariant divisor on $X$. The \textit{Newton polygon} of $D$ is 
\[ P_D = \{ m\in M_\R \hspace{1ex}|\hspace{1ex} d_i+\langle m,v_i\rangle \geq 0 \text{ for all } i=1,\dots,s \}\]
A key fact is that the lattice points in $P_D$ index a basis of $T$-equivariant functions for the global sections of the line bundle $\Oo(D)$,
\[ H^0(X,\Oo(D)) = \bigoplus_{m\in P_D\cap M}\C\cdot \chi^m. \]
Indeed, $\div(\chi^m) = \sum_i \langle m,v_i\rangle D_i$, so the defining condition of $P_D$ says that the divisor $D+\div(\chi^m)$ has nonnegative order of vanishing along each of the $T$-equivariant divisors $D_1,\dots,D_s$. It follows from the definition that changing $D$ within its linear equivalence class translates the Newton polytope accordingly, $P_{D+\div(\chi^m)}=P_D-m$.\\

To relate these Newton polytopes to Newton-Okounkov bodies, we choose coordinates on $X$. Let $\sigma$ be the two dimensional cone in $\Sigma$ with boundary rays spanned by $v_1$ and $v_2$. Define $m_1$ and $m_2$ to be primitive generators of the dual cone $\sigma^\vee$, with $\langle m_i,v_j\rangle =\delta_{ij}$ for $i,j=1,2$. Since we assumed $X$ to be smooth, $m_1$ and $m_2$ form a $\Z$-basis of $M$. We denote the coordinates on $U_\sigma\simeq \C^2$ by $x$ and $y$, so that the torus character $\chi^{pm_1+qm_2}|_{U_\sigma}$ corresponds to $x^py^q$. This implies that $D_1|_{U_\sigma}$ is defined by $x=0$, $D_2|_{U_\sigma}$ is defined by $y=0$, and $D_3,\dots,D_s$ are all disjoint from $U_\sigma$.\\

There is a short exact sequence,
\[ \begin{tikzcd}
0 \arrow[r] & M \arrow[r] & \Z^s \arrow[r] & \Pic(X) \arrow[r] & 0,
\end{tikzcd} \] 
and $\Pic(X)$ is torsion-free. Here $\Z^s$ is the set of $T$-invariant divisors generated by $D_1,\dots,D_s$, the first map sends a lattice point $m$ to the principle divisor $\div(\chi^m) = \sum_{i=1}^s\langle m,v_i \rangle D_i$, and the second map sends a divisor to its class. These results imply that $\Pic(X)$ is isomorphic to $\Z^{s-2}$, freely generated by the divisor classes $D_3,\dots,D_s$. \\

For the rest of the paper we identify $M\simeq \Z^2$ with generators $m_1$ and $m_2$ corresponding to our choice of open set $U_\sigma$. We use coordinates $(a,b)$ to denote $am_1+bm_2$ in either $M\simeq\Z^2$ or $M_\R\simeq \R^2$. We also identify $P_D$ with its image in $\R^2$, writing
\[ H^0(X,\Oo(D)) \simeq \bigoplus_{(p,q)\in P_D\cap \Z^2} \C\cdot x^py^q. \]
Each divisor class in $\Pic(X)$ has a unique representative of the form $D = \sum_{i=3}^sd_iD_i$. For such divisors, the inequalities on $P_D$ corresponding to $d_1=d_2=0$ impose the conditions $a,b\geq 0$ on points $(a,b)\in P_D\subseteq \R^2$. This polygon can therefore be defined as
\[ P_D = \left\{ (a,b)\in \R^2 \hspace{1ex}\left|\hspace{1ex} 
 \begin{varwidth}{10 in} \setstretch{1.5}
$0\leq a \leq c$, and\\
$\ell(a) \leq b \leq u(a)$
 \end{varwidth} 
 \right.
\right\}\]
for some constant $c$, and piecewise linear functions $\ell$ and $u$. This convex set is equal to the Newton-Okounkov body of $D$, with valuation given by the order of vanishings along $D_1$ and $D_2$, as explained in \cite{LM}. We abuse terminology by referring to this Newton-Okounkov body as the Newton polytope of the entire class of $D$ in $\Pic(X)$.

\subsection{Hilbert Schemes of Points on Surfaces}\label{Hilbbg}

Let $X$ be a smooth, irreducible surface over $\C$, and $n\geq 2$ an integer. We assume that the irregularity of the surface vanishes $q(X) = h^1(\Oo_X)=0$, which is the case when $X$ is a toric surface. The Hilbert scheme of $n$ points on $X$, denoted $X^{[n]}$, parametrizes zero-dimensional subschemes of $X$ of length $n$. The simplest such subschemes are the reduced subschemes supported on $n$ distinct points of $X$, which is why $X^{[n]}$ is referred to as the Hilbert scheme of points on $X$.\\

Let $X^{(n)} = X^n/S_n$ denote the symmetric power of $X$. The Hilbert scheme $X^{[n]}$ comes equipped with the Hilbert-Chow morphism $X^{[n]}\to X^{(n)}$, which maps a length $n$ subscheme $Z\subseteq X$ to its support counted with multiplicity. A fundamental result of Fogarty states that the Hilbert scheme $X^{[n]}$ is a smooth, irreducible, variety of dimension $2n$ \cite{Fo1}. Furthermore, $X^{[n]}$ is projective whenever the surface $X$ is projective. The Hilbert-Chow morphism $ X^{[n]}\to X^{(n)} $ is a birational resolution of singularities, and an isomorphism over the dense open set of reduced subschemes $U\subseteq X^{[n]}$.\\

When $X$ is a toric surface, $X^{[n]}$ is not typically a toric variety, but it does inherit a two-dimensional torus action from that on $X$. This is the diagonal torus action, where $t\in T$ sends the reduced subscheme supported on $\{p_1,\dots,p_n\}$ to the reduced subscheme supported on $\{t\cdot p_1,\dots,t\cdot p_n\}$.  \\

In \cite{Fo2}, Fogarty constructs a linear embedding of divisors $\Pic(X)\hookrightarrow \Pic(X^{[n]})$ as follows: A divisor $D\in \Pic(X)$ determines a symmetric divisor $\pi_1^*D+\cdots +\pi_n^*D$ on $X^n$, where $\pi_1,\dots,\pi_n$ are the coordinate projections $X^n\to X$. By symmetry this divisor descends to a divisor $D_{(n)}$ on the symmetric power $X^{(n)}$. The divisor $D_n\in \Pic(X^{[n]})$ is defined to be the pullback of $D_{(n)}$ via the Hilbert-Chow morphism. The exceptional divisor of the Hilbert-Chow morphism, the locus of nonreduced schemes, is an irreducible divisor which we denote by $B$. As noted in the introduction, we often express divisors in terms of the class $E=-\frac12B$ rather than $B$, which is purely a notational choice. Under the hypotheses $q(X)=0$, Fogarty shows that there is an isomorphism
\[ \Pic(X^{[n]})\simeq \Pic(X)\times \Z\cdot E. \]

\subsection{An Algebraic Model of $(\C^2)^{[n]}$}\label{algsection}

In this section we mainly follow Haiman's study of the Hilbert scheme $(\C^2)^{[n]}$ from \cite{Ha2}. Consider the polynomial ring  $\C[{\bf x, y}]= \C[x_1,\dots,x_n,y_1,\dots,y_n],$ with the diagonal action of the symmetric group $S_n$. This means that $S_n$ permutes the pairs $(x_1,y_1),\dots,(x_n,y_n)$ in blocks, i.e.
\[ w x_i = x_{w(i)},\hspace{1cm}w y_i = y_{w(i)} \hspace{1cm} \text{for all }w\in S_n. \]

With respect to this diagonal action, let $A^0 = \C[{\bf x, y}]^{S_n}$ be the space of symmetric polynomials, and $A^1 = \C[{\bf x, y}]^{\epsilon}$ the space of alternating polynomials. For $r>1$, define $A^r$ to be the linear span of all $r$-fold products of alternating polynomials. These spaces can be assembled into a graded ring, which we denote by $S = A^0\oplus A^1 \oplus A^2 \oplus \cdots$. 

\begin{theorem}[Haiman \cite{Ha2}]
The Hilbert scheme of points $(\C^2)^{[n]}$, equipped with the Hilbert-Chow morphism, is isomorphic to $\Proj(S)$ as a scheme over the symmetric power $(\C^2)^{(n)} = \Spec A^0$.
\end{theorem}

The above isomorphism equips $(\C^2)^{[n]}$ with an ample line bundle $\Oo_{(\C^2)^{[n]}}(1)$, or simply $\Oo(1)$. The line bundle $\Oo(1)$ is linearly equivalent to $\Oo(E) \simeq \Oo(-\frac12B)$, where $B$ is the divisor of nonreduced subschemes of $\C^2$. The space of global sections $H^0((\C^2)^{[n]},\Oo(r))$ is isomorphic to the degree $r$ piece of the integral closure of $S$ (\cite{Hart}, Section 2 Ex. 5.14). In fact $S$ is already integrally closed, a result which we later deduce from Haiman's results, along with our study of valuations (Corollary \ref{int}). This may be known to experts, but we have been unable to locate a reference.\\

Let $J\subseteq \C[\bf{x,y}]$ denote the ideal generated by $A^1$. Define $\overline{S} = \overline{A}^0\oplus \overline{A}^1\oplus \overline{A}^2\oplus \cdots$ where $\overline{A}^{2r} = A^{0}\cap J^{2r}$ and $\overline{A}^{2r+1} = A^{1}\cap J^{2r+1}$ for all $r\geq 0$. 

\begin{lemma}\label{bounds}
For all $r\geq0$, we have
\[A^r\subseteq H^0((\C^2)^{[n]},\Oo(r))\subseteq \overline{A}^r \] 
In the cases $r=0,1$ there is equality $A^0=\overline{A}^0$, and $A^1=\overline{A}^1$.
\end{lemma}

\begin{proof}
The equalities $A^0=\overline{A}^0$ and $A^1=\overline{A}^1$ follow from the definitions, and the inclusion $A^r\subseteq \overline{A}^r$ is also straightforward. Since $H^0((\C^2)^{[n]},\Oo(r))$ is the degree $r$ part of the integral closure of $S$, it suffices to show that $\overline{S}$ is integrally closed.\\

Using his proof of the Polygraph Theorem, Haiman (Proposition 4.3 of \cite{Ha1}) proves that $J$ is equal to the radical ideal
\[ J = \bigcap_{i\neq j}(x_i-x_j,y_i-y_j), \] 
and more generally
\[ J^r = \bigcap_{i\neq j}(x_i-x_j,y_i-y_j)^r \]
for all $r\geq 1$. This implies that the ring $\C[\mathbf{x,y}][tJ]\subseteq \C[\mathbf{x,y},t]$ is integrally closed.\\

The following argument is also due to Haiman (\cite{Ha2}, p. 218): Consider the action of $S_n$ on $\C[\mathbf{x,y},t]$ extending the diagonal $S_n$ action on $\C[\mathbf{x,y}]$ by setting $\sigma(t) = \sgn(\sigma)t$. The ring of invariants $\C[\mathbf{x,y},t]^{S_n}$ is $A^0\oplus A^1 \oplus A^0\oplus A^1\oplus\cdots$, which implies that this ring is also integrally closed.\\

These two facts imply that $\overline{S}$ is integrally closed as well, as it is the intersection of two integrally closed subrings of $\C[\mathbf{x,y},t]$, completing the proof.
\end{proof}

The line bundles $\Oo(r)$ are $T$-linear, where $T\simeq (\C^*)^2$ acts on $(\C^2)^{[n]}$ diagonally as in the previous section. The $T$-action on $\Oo(r)$ induces a $\Z^2$-grading on the global sections of $\Oo(r)$. Under the inclusion $H^0((\C^2)^{[n]},\Oo(r))\subseteq \C[\mathbf{x,y}]$, this $\Z^2$-grading is inherited from the grading on $\C[\mathbf{x,y}]$ in which the monomial $x_1^{p_1}\cdots x_n^{p_n}y_1^{q_1}\cdots y_n^{q_n}$ has degree $(p_1+\cdots+p_n,q_1+\cdots+q_n)$. In other words, the spaces of global sections $H^0((\C^2)^{[n]},\Oo(r))$ are graded subspaces of $\C[\mathbf{x,y}]$ with respect to this $\Z^2$-grading.

\section{The Hilbert Scheme of points on $\C^2$}\label{c2section}

The main goal of this section is to compute the Newton-Okounkov body of the line bundles $\Oo(r)$ on $(\C^2)^{[n]}$, defined in the previous section. This computation is based on Lemma \ref{bounds}, which identifies the global sections of $\Oo(r)$ with certain sets of polynomials. In the notation introduced in the previous section, Lemma \ref{bounds} states that
\[ A^r\subseteq H^0((\C^2)^{[n]},\Oo(r))\subseteq \overline{A}^r \subseteq \C[\mathbf{x,y}] \]
for all $r\geq0$. We show that $A^r = \overline{A}^r$ for all $r\geq 0$ (Corollary \ref{int}), characterizing the spaces of sections exactly. The main step in the proof is to show that the sets of trailing terms obtained from polynomials in $A^r$ are the same as those obtained from $\overline{A}^r$, with respect to a certain term order. We use the corresponding trailing term valuation, so this computation also provides the valuations out of which the Newton-Okounkov bodies are constructed. Finally, we use this description of the sets of valuations to compute the Newton-Okounkov bodies (Theorem \ref{noc2}). \\

More specifically, we use the following valuation.

\begin{definition}\label{val}
Let $\nu:\C[{\bf x, y}]\setminus\{0\}\to \Z^{2n}$ be the trailing term valuation in the lexicographic term order, with $x_1>x_2>\cdots>x_n>y_1>y_2>\cdots>y_n$. 
\end{definition}

The valuation $\nu$ has one-dimensional leaves, as do all leading/trailing term valuations on polynomials. We use coordinates $(a_1,\dots,a_n,b_1,\dots,b_n)$ on $\Z^{2n}$, so that the $a_i$ coordinates correspond to the exponents of the $x_i$'s, and the $b_i$ coordinates correspond to the exponents of the $y_i$'s. For example, the terms of the polynomial $f = x_1^2x_2^2y_2+x_1x_2^4y_1^5y_2^3\in \C[x_1,x_2,y_1,y_2]$ have exponent vectors $(a_1,a_2,b_1,b_2) = (2,2,0,1),$ and $(1,4,5,3)$. So $\nu(f) = (1,4,5,3)$, the smaller vector lexicographically.\\

For later reference, we define sets $\Gamma_r$, which turn out the be the sets of valuations of polynomials in $A^r$ (and therefore in $H^0((\C^2)^{[n]},\Oo(r))$ as well).

\begin{definition}\label{semigroup}
For integers $r\geq 0$, set
\[\setstretch{1.5}
\Gamma_r :=  \left\{ 
\begin{varwidth}{10 in} \setstretch{1.5}
 $(p_1,\dots,p_n,$\\
 $q_1,\dots,q_n)\in \Z^{2n}_{\geq 0}$
 \end{varwidth} 
 \hspace{1ex}\left|\hspace{1ex} 
 \begin{varwidth}{10 in} \setstretch{1.5}
$p_1\leq p_2\leq\cdots\leq p_n$,\\
if $p_{j} = p_{j+1}$ then $q_{j+1}\geq q_j+r$, and\\
$q_j \geq  \sum_{\substack{i=1,\dots,j-1 \\ p_j-p_i<r}} (r-p_j+p_i)$ for all $j$ 
 \end{varwidth} 
 \right.
 \right\}.
 \]
Equivalently,
\[ \Gamma_r = \left\{ 
\begin{varwidth}{10 in} \setstretch{1.5}
$(p_1,\dots,p_n,$\\
 $q_1,\dots,q_n)\in \Z^{2n}_{\geq 0}$
 \end{varwidth} 
 \hspace{1ex}\left|\hspace{1ex} 
 \begin{varwidth}{10 in} \setstretch{1.5}
$p_1\leq p_2\leq\cdots\leq p_n$,\\
if $p_{j} = p_{j+1}$ then $q_{j+1}\geq q_j+r$, and\\
$q_j\geq (j-i)(r-p_j)+p_i+\cdots +p_{j-1} $,\\
for all $1\leq i \leq j \leq n$
 \end{varwidth} 
 \right.
 \right\}.
  \setstretch{1}
\]
\end{definition}
These descriptions are indeed equivalent, because $p_1\leq \cdots\leq p_n$ implies that
\begin{align*}
\sum_{\substack{\ell=1,\dots,j-1 \\ p_j-p_\ell<r}} (r-p_j+p_\ell) & = \max_{i=1,\dots,j}\left\{ \sum_{\ell=i}^{j-1} (r-p_j+p_\ell) \right\}\\
& =  \max_{i=1,\dots,j} \left\{ (j-i)(r-p_j)+p_i+\cdots +p_{j-1} \right\}
\end{align*}
for all $j$, as the first sum is over precisely the terms $(r-p_j+p_\ell)$ which are positive.\\

In Sections \ref{01} and \ref{semigroupsection} we give simpler descriptions of the sets $\Gamma_r$, but it is convenient to take the explicit description above as the definition.

\subsection{Bases for $A^0$ and $A^1$}\label{01}

The spaces $A^0$ and $A^1$ have well-known bases that realize $\Gamma_0$ and $\Gamma_1$ as their respective sets of valuations.\\

For any $n$-tuple $(\mathbf{p,q}) = ((p_1,q_1),\dots,(p_n,q_n))\in (\Z_{\geq 0}^2)^n$, possibly with repetitions, there is a monomial symmetric polynomial
\[ m_{(\mathbf{p,q})}(\mathbf{x,y}) = x_1^{p_1}\cdots x_n^{p_n}y_1^{q_1}\cdots y_n^{q_n} + (\text{symmetric terms}) \in A^0. \]
The polynomials $m_{(\mathbf{p,q})}(\mathbf{x,y}) $ form a linear basis for the space of all symmetric polynomials $A^0$ as $(\mathbf{p,q})$ ranges over all such $n$-tuples up to reordering. We label $(\mathbf{p,q})$ so that $(p_1,q_1)\leq (p_2,q_2)\leq \cdots \leq (p_n,q_n)$ in lexicographic order, which implies that 
\[\nu(m_{(\mathbf{p,q})}(\mathbf{x,y})) = (p_1,\dots,p_n,q_1,\dots,q_n).\]
Since each basis element has a different valuation, these are all of the valuations that can be obtained from $A^0$. In other words, we have
\[ \{ \nu(f) \hspace{1ex}|\hspace{1ex} f\in A^0\setminus \{0\} \} = \left\{ (p_1,\dots,p_n,q_1,\dots,q_n)\in \Z^{2n}_{\geq0} \hspace{1ex}\left|\hspace{1ex}
 \begin{varwidth}{10 in} \setstretch{1.5}
$(p_1,q_1)\leq \cdots \leq (p_n,q_n)$\\
in lex
 \end{varwidth} 
 \right.
 \right\} .
 \]
 
Similarly, for any $n$-tuple $(\mathbf{p,q}) = ((p_1,q_1),\dots,(p_n,q_n))\in (\Z_{\geq 0}^2)^n$ of distinct pairs there is an alternating polynomial
\[ d_{(\mathbf{p,q})}(\mathbf{x,y}) = \det(x_i^{p_j}y_i^{q_j})_{1\leq i,j\leq n} \in A^1. \]
The determinants $ d_{(\mathbf{p,q})}(\mathbf{x,y}) $ form a linear basis for the space of all alternating polynomials $A^1$ as $(\mathbf{p,q})$ ranges over all such $n$-tuples up to reordering. Just as before, we label $(\mathbf{p,q})$ so that $(p_1,q_1)< (p_2,q_2)< \cdots < (p_n,q_n)$ in lexicographic order, which implies that \[ \nu(d_{(\mathbf{p,q})}(\mathbf{x,y})) = (p_1,\dots,p_n,q_1,\dots,q_n). \]
By the same argument, we therefore have
\[ \{ \nu(f) \hspace{1ex}|\hspace{1ex} f\in A^1\setminus \{0\} \} = \left\{ (p_1,\dots,p_n,q_1,\dots,q_n)\in \Z^{2n}_{\geq0} \hspace{1ex}\left|\hspace{1ex}
 \begin{varwidth}{10 in} \setstretch{1.5}
$(p_1,q_1)< \cdots < (p_n,q_n)$\\
in lex
 \end{varwidth} 
 \right.
 \right\} .
 \]

\begin{lemma}\label{alt01}
\[ \Gamma_0 = \left\{ (p_1,\dots,p_n,q_1,\dots,q_n)\in \Z^{2n}_{\geq0} \hspace{1ex}\left|\hspace{1ex}
 \begin{varwidth}{10 in} \setstretch{1.5}
$(p_1,q_1)\leq \cdots \leq (p_n,q_n)$\\
in lex
 \end{varwidth} 
 \right.
 \right\} , \]
and 
\[ \Gamma_1 = \left\{ (p_1,\dots,p_n,q_1,\dots,q_n)\in \Z^{2n}_{\geq0} \hspace{1ex}\left|\hspace{1ex}
 \begin{varwidth}{10 in} \setstretch{1.5}
$(p_1,q_1)< \cdots < (p_n,q_n)$\\
in lex
 \end{varwidth} 
 \right.
 \right\} . \]
\end{lemma}

\begin{proof}
In both cases, we study the inequalities 
\[q_j \geq  \sum_{\substack{i=1,\dots,j-1 \\ p_j-p_i<r}} (r-p_j+p_i) \]
for $j=1,\dots,n$ appearing in the definition of $\Gamma_r$. When $r=0$ the inequalities $p_1\leq p_2\leq\cdots\leq p_n$ imply that $p_j-p_i<r$ is never satisfied. These inequalities therefore reduce to $q_j\geq 0$ for all $j$. One checks that the remaining conditions defining $\Gamma_0$ precisely describe the set in the statement of the lemma.\\

Similarly, when $r=1$ the condition $p_j-p_i<r$ is equivalent to $p_i=p_j$, so the inequalities above reduce to
\[ q_j\geq \# \{ i=1,\dots,j-1 \hspace{1ex}|\hspace{1ex} p_i=p_j \}.\] 
This condition is redundant though, because the other conditions imply that $0\leq q_i<q_{i+1}<\cdots <q_j$ whenever $p_i=p_{i+1}= \cdots =p_j$. As before, the remaining conditions defining $\Gamma_1$ precisely describe the set in the statement of the lemma.
\end{proof}

Finally, we note that each polynomial $d_{(\mathbf{p,q})}(\mathbf{x,y})$ and $m_{(\mathbf{p,q})}(\mathbf{x,y})$ is homogeneous with respect to the $\Z^2$-grading on $\C[\mathbf{x,y}]$ defined in Section \ref{algsection}. 

\subsection{Valuations from $\overline{A}^r$ for $r>1$}\label{upper}

We aim to show that for any $r>1$, $\Gamma_r$ contains the valuations $\nu(f)$ of all nonzero polynomials $f\in \overline{A}^r$. We prove a slightly more general result, Proposition \ref{np}, relating the Newton polytope of a polynomial $f\in A^r$ to its valuation. This general version has essentially the same proof, and is used in Section \ref{projsection} when we turn to the general toric surface case. For brevity, we say that a polynomial $f$ is (anti-)symmetric if it is either symmetric or alternating with respect to a specified symmetric group action.

\begin{lemma}\label{divis}
Let $f$ be a polynomial in $x,x',y,y'$, and possibly other variables. Suppose that $f$ is (anti-)symmetric with respect to the $S_2$ action exchanging $x$ and $y$ with $x'$ and $y'$. Additionally, suppose that $f$ is contained in the ideal
\[ I^r  = (x-x',y-y')^r = ((x-x')^r, (x-x')^{r-1}(y-y'),\dots,(y-y')^r) \]
for some fixed $r\geq0.$ Let $p$ denote the largest integer such that $x^p|f$. For any fixed $p' = p,p+1,\dots, p+r-1$, the polynomial $[x^p(x')^{p'}]f$ is divisible by $(y-y')^{r-p'+p}$.
\end{lemma}

Here we use the coefficient extraction operator: $[x^p(x')^{p'}]f$ is defined to be the polynomial obtained by summing all the terms of $f$ whose exponents on $x$ and $x'$ are exactly $p$ and $p'$ respectively, then dividing by the common factor $x^p(x')^{p'}$.

\begin{proof}
By assumption, $f$ is divisible by $x^p$, so by (anti-)symmetry it is also divisible by $(x')^p$. Dividing $f$ by $(xx')^p$, we may assume without loss of generality that $p=0$. \\

The assumption $f\in I^r$ implies that, replacing $x'$ by $x$ in the expression for $f$, we can write $f|_{x'=x}= (y-y')^r g_0$ for some polynomial $g_0$ which does not depend on $x'$. It follows that $f - (y-y')^r g_0$ is divisible by $(x-x')$, and we define $f_1$ to be the polynomial such that $f = (y-y')^{r}g_0+(x-x') f_1$. This $f_1$ is contained in the ideal $I^{r-1}$, so we can repeat this argument to get an expression 
\[ f =\sum_{k=0}^r (x-x')^k(y-y')^{r-k}g_k \]
in which none of the polynomials $g_0,\dots,g_{r-1}$ depend on $x'$. It follows that
\[ [x^0]f = f|_{x=0} =\sum_{k=0}^r (x')^k(y-y')^{r-k}\overline g_k \]
where $\overline g_k$ is obtained from $g_k$ by setting $x=0$, and possibly changing the sign. But since none of the polynomials $\overline g_0,\dots,\overline g_{r-1}$ depend on $x'$, the terms of $[x^0]f$ are grouped in this sum according to their exponent on $x'$. One sees from this expression that $[x^0(x')^{k}]f = (y-y')^{r-k}\overline g_k$ for all $k=0,\dots, r-1$, as desired.
\end{proof}

The Newton polytope of a polynomial in $d$ variables is the convex hull in $\R^d$ of the exponent vectors of the nonzero terms of the polynomial. The Newton polytope of a product of polynomials is the Minkowski sum of the Newton polytopes of its factors. Lemma \ref{divis} allows us to establish a lower bound for the Newton polytope of a polynomial in $\overline{A}^r$ with a given valuation.

\begin{proposition}\label{np}
For any $r>1$, let $f\in \overline{A}^r$ be a polynomial with valuation $(p_1,\dots,p_n,q_1,\dots,q_n)$. For each $j=1,\dots,n$, there exist points in the Newton polytope of $f$ whose $(a_j,b_j)$ coordinates are
\[ \left( p_j, q_j -  \sum_{\substack{i=1,\dots,j-1 \\ p_j-p_i<r}} (r-p_j+p_i)\right) \text{, and } \left( p_j, q_j + \sum_{\substack{k = j+1,\dots,n \\ p_k-p_j<r}} (r-p_k+p_j)\right). \]
Furthermore, for any $j=1,\dots,n-1$ such that $p_j=p_{j+1}$, we have $q_{j+1}\geq q_j+r$.
\end{proposition}

As usual, we use coordinates $(a_1,\dots,a_n,b_1,\dots,b_n)$ on $\Z^{2n}$ so that the $a_j$-coordinate corresponds to the exponent on $x_j$, and the $b_j$-coordinate corresponds to the exponent on $y_j$.

\begin{corollary}\label{gammaub}
For any $r>1$ and nonzero $f\in \overline{A}^r$, we have $\nu(f)\in\Gamma_r$.
\end{corollary}

\begin{proof}[Proof of Corollary \ref{gammaub}]
Let $f\in \overline{A}^r$ be a nonzero polynomial. By Definition \ref{val}, $\nu(f) = (p_1,\dots,p_n,q_1,\dots,q_n)$ is in $\Gamma_r$ if and only if
\begin{enumerate}
\item $p_1\leq p_2\leq\cdots\leq p_n$,
\item $p_{j} = p_{j+1}$ implies that $q_{j+1}\geq q_j+r$, and
\item $q_j \geq  \sum_{\substack{i=1,\dots,j-1 \\ p_j-p_i<r}} (r-p_j+p_i)$ for all $j=1,\dots,n$. 
\end{enumerate}

Also by definition, we have $\overline{A}^r\subseteq A^0$ for even $r$, and $\overline{A}^r\subseteq A^1$ for odd $r$. By Lemma \ref{alt01} we have $(p_1,q_1)\leq \cdots \leq (p_n,q_n)$ in lex, and in particular $p_1\leq p_2\leq \cdots \leq p_n.$ The second condition is explicitly stated to hold in Proposition \ref{np}. Finally, since $f$ is a polynomial, the coordinates of any point in its Newton polytope are nonnegative. Proposition \ref{np} therefore implies that
\[ q_j \geq  \sum_{\substack{i=1,\dots,j-1 \\ p_j-p_i<r}} (r-p_j+p_i) \]
for all $j=1,\dots,n$, as desired.
\end{proof}

\begin{proof}[Proof of Proposition \ref{np}]
Let $g = g(y_1,\dots, y_n)$ be the polynomial $[x_1^{p_1}\cdots x_n^{p_n}]f$. For any two indices $1\leq i < k \leq n$ such that $p_k-p_i<r$, let \[ g_{ik} = g_{ik}(x_i,x_k,y_1,\dots,y_n) = [x_1^{p_1}\cdots \hat x_i\cdots \hat x_k \cdots x_n^{p_n}]f.\] By hypothesis, the lex trailing term of $f$ is $x_1^{p_1}\cdots x_n^{p_n}y_1^{q_1}\cdots y_n^{q_n}$, so the lex trailing term of $g$ is $y_1^{q_1}\cdots y_n^{q_n}$ and the lex trailing term of $g_{ik}$ is $x_i^{p_i}x_k^{p_k}y_1^{q_1}\cdots y_n^{q_n}.$ Since $f\in \overline A^r$ by assumption, $g_{ik}$ is (anti-)symmetric in the pairs of variables $(x_i,y_i)$ and $(x_k,y_k)$ and is contained in the ideal $(x_i-x_k,y_i-y_k)^r$. We apply Lemma \ref{divis} to $g_{ik}$, using $(x_i,y_i)$ and $(x_k,y_k)$ for $(x,y)$ and $(x',y')$ respectively. In the notation of Lemma \ref{divis} we have $p=p_i$, and choose $p' = p_k$. The lemma then says that $(y_i-y_k)^{r-p_k+p_i}$ divides $g = [x_i^{p_i}x_k^{p_k}]g_{ik}$.\\

Since $g$ is divisible by these factors for all such $i$ and $k$, we have
\[ g(y_1,\dots,y_n) = h(y_1,\dots,y_n)\prod_{\substack{1\leq i<k\leq n \\ p_k-p_i<r}}(y_i-y_k)^{r-p_k+p_i} \]
for some polynomial $h(y_1,\dots,y_n)$. The lex trailing term of $g$ is $y_1^{q_1}\cdots y_n^{q_n}$, and since the trailing term of each factor $(y_i-y_k)^{r-p_k+p_i}$ is $y_k^{r-p_k+p_i}$, the trailing term of $h$ is $y_1^{q_1'}\cdots y_n^{q_n'}$, where $q_j'$ satisfies
\[ q_j = q_j' + \sum_{\substack{i=1,\dots,j-1 \\ p_j-p_i<r}} (r-p_j+p_i) \]
for each $j=1,\dots,n$.\\

For any $i<k$ such that $ p_k-p_i<r$, let $\Delta_{ik}$ be the Newton polytope of $(y_i-y_k)^{r-p_k+p_i}$, which is the convex hull of the two vectors in the directions of the coordinates $b_i$ and $b_k$ with lengths $(r-p_k+p_i)$. By the expression of $g$ as a product, and the definition of $g$, the Newton polytope of $f$ contains the Minkowski sum
\[ \Delta := \{ (p_1,\dots,p_n,q_1',\dots,q_n') \} + \sum_{\substack{1\leq i<k\leq n \\ p_k-p_i<r}} \Delta_{ik}. \]

For each $j=1,\dots,n$, the $a_j$ coordinate is constant equal to $p_j$ over the whole set $\Delta$, so we study the maximum and minimum values of the $b_j$ coordinates. The minimum value of the $b_j$ coordinate on $\Delta$ is
\[ q_j' = q_j -  \sum_{\substack{i=1,\dots,j-1 \\ p_j-p_i<r}} (r-p_j+p_i), \] 
obtained by taking points from all the $\Delta_{ik}$'s whose $b_j$ coordinate is zero. Taking points in $\Delta_{ik}$ whose $b_j$ coordinate is as large as possible shows that the the maximum value of $b_j$ coordinate on $\Delta$ is
\[ q_j'+ \sum_{\substack{i=1,\dots,j-1 \\ p_j-p_i<r}} (r-p_j+p_i) + \sum_{\substack{k = j+1,\dots,n \\ p_k-p_j<r}} (r-p_k+p_j) = q_j + \sum_{\substack{k = j+1,\dots,n \\ p_k-p_j<r}} (r-p_k+p_j). \]
This proves the first claim of the proposition.\\

As for the second claim, fix some $j=1,\dots,n-1$ such that $p_j=p_{j+1}$. By hypothesis, $f$ is (anti-)symmetric under swapping $(x_j,y_j)$ with $(x_{j+1},y_{j+1})$. These assumptions imply that
\[ g(y_1,\dots,y_n) = h(y_1,\dots,y_n)\prod_{\substack{1\leq i<k\leq n \\ p_k-p_i<r}}(y_i-y_k)^{r-p_k+p_i} \]
is (anti-)symmetric in the variables $y_j,y_{j+1}$, as $g$ consists of terms of $f$ whose exponents on $x_j$ and $x_{j+1}$ are equal. Swapping $y_j$ and $y_{j+1}$ in the factored expression for $g$ fixes the pairs of factors
\[ (y_i-y_j)^{r-p_j+p_i}(y_i-y_{j+1})^{r-p_{j+1}+p_i} \]
for $i<j$ and 
\[ (y_j-y_k)^{r-p_k+p_j}(y_{j+1}-y_k)^{r-p_k+p_{j+1}} \]
for $k>j+1$. The factors not including $y_j$ or $y_{j+1}$ are unaffected by the exchange, and the remaining factor $(y_j-y_{j+1})^r$ is multiplied by $(-1)^r$. This implies that $h$ is (anti-)symmetric in $y_j$ and $y_{j+1}$, so we have $q_j' \leq q_{j+1}'$. All the pairs of factors above contribute equally to the exponents on $y_j$ and $y_{j+1}$ in the trailing term of $g$. The remaining factor $(y_j-y_{j+1})^r$ has lex trailing term $y_{j+1}^r$, which implies that $q_{j+1}\geq q_j+r$ as desired.
\end{proof}

\subsection{An Alternate Description of $\Gamma_r$}\label{semigroupsection}

In Section \ref{01} we showed that $\Gamma_0$ and $\Gamma_1$ admit simple descriptions in terms of nondecreasing (resp. strictly increasing) $n$-tuples of points. In this section, we characterize the remaining sets $\Gamma_r$ for $r>1$ in terms of $\Gamma_1.$

\begin{proposition}\label{fg}
For $r>1$, $\Gamma_r$ is equal to the $r$-fold Minkowski sum $\Gamma_1+\cdots+\Gamma_1$.
\end{proposition}

\begin{proof}
For the containment $\Gamma_1+\cdots +\Gamma_1\subseteq \Gamma_r$, we use the second description of $\Gamma_r$ given in Definition \ref{semigroup}. Let $(\mathbf{p}^{(1)},\mathbf{q}^{(1)}),\dots,(\mathbf{p}^{(r)},\mathbf{q}^{(r)})\in \Gamma_1$ with $(\mathbf{p}^{(\ell)},\mathbf{q}^{(\ell)}) = (p^{(\ell)}_1,\dots,p^{(\ell)}_n,q^{(\ell)}_1,\dots,q^{(\ell)}_n)$, and define 
\[ \mathbf{(p,q)} = (p_1,\dots,p_n,q_1,\dots,q_n) = (\mathbf{p}^{(1)},\mathbf{q}^{(1)})+\cdots+(\mathbf{p}^{(r)},\mathbf{q}^{(r)}).\]
The inequalities $p_1\leq \cdots \leq p_n$ and $q_j\geq (j-i)(r-p_j)+p_i+\cdots +p_{j-1} $ are immediate since these conditions are homogeneous in $r$. If $p_j=p_{j+1}$, then we also have $p_j^{(\ell)} = p_{j+1}^{(\ell)}$ for all $\ell$, and therefore 
\[ q_{j+1} = q_{j+1}^{(1)}+\cdots+q_{j+1}^{(r)} \geq (q_j^{(1)}+1)+\cdots+(q_j^{(r)}+1) = q_j+r, \]
as desired. This completes the first containment.\\

For the containment $\Gamma_r \subseteq \Gamma_1+\cdots +\Gamma_1$, let $\mathbf{(p,q)} = (p_1,\dots,p_n,q_1,\dots,q_n) \in \Gamma_r$. We aim to construct an $r$-tuple $(\mathbf{p}^{(1)},\mathbf{q}^{(1)}),\dots,(\mathbf{p}^{(r)},\mathbf{q}^{(r)})\in \Gamma_1$ whose sum is $\mathbf{(p,q)}$, but first we make some reductions. \\

Reduction 1: Suppose there is an index $j=1,\dots,n-1$ such that $p_{j+1}>p_j+r$. By the first description of $\Gamma_r$ in Definition \ref{semigroup}, we have
\[ (\mathbf{p',q'}) =  (p_1,\dots,p_j,p_{j+1}-1,\dots,p_n-1,q_1,\dots,q_n) \in \Gamma_r, \]
If we have $(\mathbf{p}^{(1)},\mathbf{q}^{(1)}),\dots,(\mathbf{p}^{(r)},\mathbf{q}^{(r)})\in \Gamma_1$ whose sum is $(\mathbf{p',q'})$ then we may take any of these, say $(\mathbf{p}^{(1)},\mathbf{q}^{(1)})$, and replace it by 
\[ (p^{(1)}_1,\dots,p^{(1)}_j,p^{(1)}_{j+1}+1,\dots,p^{(1)}_n+1,q^{(1)}_1,\dots,q^{(1)}_n)\in \Gamma_1. \] 
The new $r$-tuple sums to $(\mathbf{p,q})$, which shows that if $(\mathbf{p',q'})$ is in the $r$-fold Minkowski sum, then $(\mathbf{p,q})$ is as well. It therefore suffices to consider only those $(\mathbf{p,q})\in \Gamma_r$ such that $p_{j+1}\leq p_j+r$ for all $j=1,\dots,n-1$. By a similar argument we may reduce to the case $p_1=0$.\\

Reduction 2: Suppose there is an index $j$ such that \[ q_j> \sum_{\substack{i=1,\dots,j-1 \\ p_j-p_i<r}} (r-p_j+p_i).\] Let $k$ be the largest index such that $p_j=p_{j+1}=\cdots=p_k$. In this case, we subtract one from the $b_j,\dots,b_k$ coordinates of $(\mathbf{p,q})$, defining
\[ (\mathbf{p',q'}) =  (p_1,\dots,p_n,q_1,\dots,q_j-1,\dots,q_k-1,\dots,q_n). \]
One checks that $(\mathbf{p',q'})\in \Gamma_r$. Suppose we have $(\mathbf{p}^{(1)},\mathbf{q}^{(1)}),\dots,(\mathbf{p}^{(r)},\mathbf{q}^{(r)})\in \Gamma_1$ whose sum is $(\mathbf{p',q'})$. Then since $p_j=\cdots =p_k$, any index $\ell=1,\dots,r$ also has $p^{(\ell)}_j = \cdots = p^{(\ell)}_k.$ By the maximality of $k$, we either have $k=n$ or there is an index $\ell$ such that $p^{(\ell)}_k<p^{(\ell)}_{k+1}$. For such an $\ell$ (if $k=n$ then any choice of $\ell$ works), replace $(\mathbf{p}^{(\ell)},\mathbf{q}^{(\ell)})$ by
\[ (p^{(\ell)}_1,\dots,p^{(\ell)}_n,q^{(\ell)}_1,\dots,q^{(\ell)}_j+1,\dots,q^{(\ell)}_k+1,\dots,q^{(\ell)}_n). \]
Again, one checks that this new vector still lies in $\Gamma_1$, and that the resulting $r$-tuple now sums to $\mathbf{(p,q)}$. This shows that it suffices to consider $\mathbf{(p,q)}$ such that  \[ q_j= \sum_{\substack{i=1,\dots,j-1 \\ p_j-p_i<r}} (r-p_j+p_i)\] for all $j=1,\dots,n.$\\

The remaining cases are covered by the following lemma.
\end{proof}

\begin{lemma}
Let $\mathbf{(p,q)} = (p_1,\dots,p_n,q_1,\dots,q_n) \in \Gamma_r$ be such that 
\begin{itemize}
\item $p_1=0$,
\item $p_{j+1}\leq p_j+r$ for all $j=1,\dots,n-1$, and 
\item $q_j =  \sum_{\substack{i=1,\dots,j-1 \\ p_j-p_i<r}} (r-p_j+p_i)$ for all $j=1,\dots, n$.
\end{itemize}
There exist $(\mathbf{p}^{(1)},\mathbf{q}^{(1)}),\dots,(\mathbf{p}^{(r)},\mathbf{q}^{(r)})\in \Gamma_1$ whose sum is $\mathbf{(p,q)}$, and such that $q^{(k)}_n = \#\{ i=1,\dots,n-1 \hspace{1ex} | \hspace{1ex} p_n-p_i<r-k+1 \}$ for all $k=1,\dots,r$.
\end{lemma}

The condition specifying the $q_n^{(k)}$ coordinates is not important to the result, but recording this extra information helps with the induction step. 

\begin{proof}
The proof is by induction on $n$. The case $n=1$ is trivial, as the conditions imply that $p_1= q_1=0$. In this case we take $(p^{(k)}_1,q^{(k)}_1) = (0,0)$ for all $k=1,\dots,n$.\\

Now let $\mathbf{(p,q)} = (p_1,\dots,p_n,q_1,\dots,q_n) \in \Gamma_r$ be as in the statement of the lemma. Apply the inductive hypothesis to $(p_1,\dots,p_{n-1},q_1,\dots,q_{n-1})\in\Gamma_r\subseteq \Z^{2n-2}$ to obtain $(p^{(k)}_1,\dots,p^{(k)}_{n-1},q^{(k)}_1,\dots,q^{(k)}_{n-1})\in \Gamma_1\subseteq \Z^{2n-2}$ for $k=1,\dots,r$. Define $\ell = p_n-p_{n-1}$, and note that $\ell \in \{0,1,\dots,r\}$ by assumption. We extend these tuples by defining
\[ (p^{(k)}_n,q^{(k)}_n) := \begin{cases}
(p_{n-1}^{(k)}+1,0) & \text{for }k=1,\dots,\ell, \text{ and} \\
(p_{n-1}^{(k)},q_{n-1}^{(k)}+1) & \text{for } k= \ell+1,\dots,r.
\end{cases} \]
It follows from the definition that $(p^{(k)}_n,q^{(k)}_n)>(p^{(k)}_{n-1},q^{(k)}_{n-1})$ in lex for all $k$, so by Lemma \ref{alt01} we have  $(p^{(k)}_1,\dots,p^{(k)}_{n},q^{(k)}_1,\dots,q^{(k)}_{n})\in \Gamma_1\subseteq \Z^{2n}$. It also follows from the construction that $p^{(1)}_n+\cdots +p^{(r)}_n = p^{(1)}_{n-1}+\cdots +p^{(r)}_{n-1}+\ell = p_{n-1}+\ell = p_n$. \\

The equality $q^{(1)}_n+\cdots +q^{(r)}_n = q_n$ is a consequence of the explicit formula for the $q_n^{(k)}$ terms. To meet this stronger condition, however, we must reorder the vectors as 
\[ (\mathbf{p}^{(\ell+1)},\mathbf{q}^{(\ell+1)}),\dots,(\mathbf{p}^{(r)},\mathbf{q}^{(r)}),(\mathbf{p}^{(1)},\mathbf{q}^{(1)}),\dots,(\mathbf{p}^{(\ell)},\mathbf{q}^{(\ell)}). \]
For $k=\ell+1,\dots,r$ corresponding to the indices $k-\ell = 1,\dots,r-\ell$ in the reordering above, we have
\begin{align*}
q^{(k)}_n = 1+ q_{n-1}^{(k)} & = 1+ \#\{ i=1,\dots,n-2 \hspace{1ex} | \hspace{1ex} p_{n-1}-p_i<r-k+1 \} \\
& =  \#\{ i=1,\dots,n-1 \hspace{1ex} | \hspace{1ex} p_n-p_i<r-(k-\ell)+1 \}.
\end{align*}
For $k=1,\dots,\ell$ corresponding to the indices $r-\ell+k = r-\ell+1,\dots,r$ in the reordering above, we have
\[ q^{(k)}_n = 0 = \#\{ i=1,\dots,n-1\hspace{1ex} | \hspace{1ex} p_n-p_i < r- (r-\ell+k) +1 \}. \]
These are precisely the required formulas for the $q_n^{(k)}$ coordinates for the reordered vectors, which completes the proof.
\end{proof}

\subsection{The Coordinate Ring and Graded Semigroup of $(\C^2)^{[n]}$}\label{CoordRingc2Section}

With the results from the previous sections, we can compute the sets of valuations obtained from $A^r$, $\overline{A}^r$, and $H^0((\C^2)^{[n]},\Oo(r))$.

\begin{proposition}\label{sg}
For all $r\geq 0$, we have 
\[ \Gamma_r = \{ \nu(f) \hspace{1ex}|\hspace{1ex}f\in A^r\setminus\{0\} \} =  \{ \nu(f) \hspace{1ex}|\hspace{1ex}f\in \overline{A}^r\setminus\{0\} \}. \]
\end{proposition}

\begin{proof}
The cases $r=0,1$ are proved in Section \ref{01}, so assume $r>1$. We establish the following chain of inclusions
\[\Gamma_r\subseteq \{ \nu(f) \hspace{1ex}|\hspace{1ex}f\in A^r\setminus\{0\} \} \subseteq \{ \nu(f) \hspace{1ex}|\hspace{1ex}f\in \overline{A}^r\setminus\{0\} \} \subseteq \Gamma_r. \]
The middle inclusion follows from $A^r\subseteq \overline{A}^r$, and the final inclusion is precisely the statement of Corollary \ref{gammaub}. For the first inclusion, take $\vec{v}\in \Gamma_r$. By Proposition \ref{fg}, $\vec{v}$ is in the $r$-fold Minkowski sum $\Gamma_1+\cdots+\Gamma_1$. But as shown in Section \ref{01}, every vector in $\Gamma_1$ is the valuation of some determinant $d_{(\mathbf{p}_1,\mathbf{q}_1)}(\mathbf{x,y})$. This implies that $\vec{v}$ is attained as the valuation of some $r$-fold product of these determinants, and all such products are contained in $A^r$ by definition. This establishes the chain of inclusions, completing the proof.
\end{proof}

This result, along with Lemma \ref{bounds}, already implies that $\Gamma_r$ is precisely the set of valuations obtained from $H^0((\C^2)^{[n]}\Oo(r))$. However, it also affords a proof of the following result, which clarifies the situation greatly.

\begin{corollary}\label{int}
For all $r\geq0$, we have $A^r=\overline{A}^r$. In particular, the ring $S=A^0\oplus A^1\oplus A^2\oplus \cdots$ in integrally closed, and therefore
\[ H^0((\C^2)^{[n]},\Oo(r)) \simeq A^r \]
for all $r\geq 0$.
\end{corollary}

This result may be known to experts, but we have been unable to find a reference. The equality $A^r=\overline{A}^r$ essentially comes from the fact that the sets of valuations from $A^r$ and $\overline{A}^r$ are the same. It is certainly possible, however, to have a strict inclusion of vector spaces $V\subset W$ and a valuation on $W$ such that every valuation from $W$ is obtained on $V$. For example take the subspace $\{ (x+1)f(x) \hspace{1ex}|\hspace{1ex} f\in \C[x] \}\subset \C[x]$ with the trailing term valuation. The key additional fact used in the proof of Corollary \ref{int} is the $\Z^2$-grading (defined at the end of Section \ref{algsection}) into finite-dimensional pieces in a way that is compatible with the valuation. 

\begin{proof}
The assertion $A^r= \overline{A}^r$ follows from the definitions for $r=0,1$, so we assume $r>1$. The containment $A^r\subseteq \overline{A}^r$ is clear, so for the reverse containment we fix $f\in \overline{A}^r$ a nonzero polynomial. Define $M$ to be the set of valuations $(\mathbf{p,q})\in \Gamma_r$ such that $f$ has a nonzero term in the $(p_1+\cdots+p_n,q_1+\cdots+q_n)$ graded piece. $M$ is finite because $f$ has terms from only finitely many graded pieces, and each graded piece has only finitely many possible valuations. It is also clear that $\nu(f)\in M$. \\

By Proposition \ref{sg} there is a polynomial $g\in A^r$ with $\nu(g) = \nu(f)$, and as in the proof of the proposition we may take $g$ to be an $r$-fold product of determinants $d_{(\mathbf{p},\mathbf{q})}(\mathbf{x,y})$. The critical observation is that $g$ is taken to be homogeneous with respect to the $\Z^2$-grading. There is a unique linear combination $f-\lambda g$ that cancels the common trailing terms, and we set $f'= f-\lambda g$. If $f'=0$ then we certainly have $f\in A^r$, so assume $f'\neq 0$. In this case we have $\nu(f')>\nu(f)$ in the lex order, but $\nu(f')$ is still in $M$ since $g$ was taken to be homogeneous. Repeat this process of reducing $f$ modulo $A^r$, in each step obtaining a polynomial with larger valuation in $M$. Since $M$ is finite this process terminates, giving an expression for $f$ as a linear combination of elements of $A^r$ as desired.
\end{proof}

\begin{remark}\label{fgfailure}
The bases constructed in Section \ref{01} for $A^0$ and $A^1$ provide all of the valuations of polynomials in $A^0$ and $A^1$, and we have now shown (by Propositions \ref{fg} and \ref{sg}) that the graded semigroup of $\Oo(1)$ is $\Gamma_0\oplus\Gamma_1\oplus\Gamma_2\oplus\cdots$, which is generated in degree one. For valuations coming from other term orders, the same bases still provide all the valuations from $A^0$ and $A^1$, as the basis elements don't share any common terms. However, for different choices of term order the graded semigroup of valuations can fail to be generated in degree one.
\end{remark}

One can also compute the sections and valuations of $\Oo(r)$ for negative $r$.

\begin{lemma}\label{neg}
For even $r<0$ we have $H^0((\C^2)^{[n]},\Oo(r))\simeq A^0$, and for odd $r<0$ we have $H^0((\C^2)^{[n]},\Oo(r))\simeq A^1$
\end{lemma}

The main part of the proof is a geometric argument identical to that of Proposition 3.2 in \cite{BC}.

\begin{proof}
Fix $n-1$ general points $p_1,\dots,p_{n-1}$ in $\C^2$, and let $R$ be the curve in $(\C^2)^{[n]}$ consisting of subschemes whose multiplicity at each of the points $p_1,\dots,p_{n-2}$ is one, and whose multiplicity at $p_{n-1}$ is two. Let $D\subseteq (\C^2)^{[n]}$ be an effective divisor linearly equivalent to $kB$ for some half integer $k>0$. The intersection product $R\cdot B= -2$ so the curve $R$ cannot meet $D$ transversely, and thus $R\subseteq D$ set theoretically. But curves of class $R$ cover a dense subset of the divisor $B\subseteq(\C^2)^{[n]}$, so there is a set theoretic inclusion $B\subseteq D$. This implies that $D-B$ is effective. \\

Since we have an isomorphism $\Oo(1)\simeq \Oo(-\frac12B)$, there are maps
\[ H^0((\C^2)^{[n]},\Oo(r+2)\to H^0((\C^2)^{[n]},\Oo(r)) \]
for all integers $r$, given by multiplication by a section defining the divisor $B$. The argument above implies that these maps are isomorphisms for all $r<0$. The desired result is obtained by composing these isomorphisms, starting from the global sections of $\Oo$ or $\Oo(1)$.
\end{proof}

\subsection{The Newton-Okounkov body of $(\C^2)^{[n]}$}\label{noc2section}

For consistency of notation we define $A^r=A^0$ for even $r<0$, and $A^r=A^1$ for odd $r<0$, and define $\Gamma_r$ similarly for $r<0$. With these conventions, the results from the previous section can be summarized as saying that
\[ H^0((\C^2)^{[n]},\Oo(r)) \simeq A^r, \]
and
\[ \Gamma_r = \{ \nu(f) \hspace{1ex}|\hspace{1ex}f\in A^r\setminus\{0\} \} \]
for all $r\in \Z.$  Following the usual construction, we therefore define the Newton-Okounkov body of $\Oo_{(\C^2)^{[n]}}(r)$ to be 
\[ \Delta(\Oo(r)) = \text{closed convex hull}\left(\bigcup_{m\geq 1}\frac{1}{m} \cdot \Gamma_{rm}\right) \]
for all $r\in \Z$.

\begin{theorem}
\label{noc2} 
For $r\geq 0$ the Newton-Okounkov body $\Delta(\Oo(r))$ is the closed convex hull of $\Gamma_r\subseteq \R^{2n}$. The Newton-Okounkov body $\Delta(\Oo)\subseteq \R^{2n}$ is a simplicial cone, and for $r<0$ we have $\Delta(\Oo) = \Delta(\Oo(r))$.
These Newton-Okounkov bodies are defined by the inequalities
\[ \setstretch{1.5}
\Delta\left(\Oo(r)\right) = \left\{ \begin{varwidth}{10 in} 
$(a_1,\dots,a_n,$\\
 $b_1,\dots,b_n)\in \R^{2n}$
 \end{varwidth}  \hspace{1ex}\left|\hspace{1ex} 
 \begin{varwidth}{10 in} \setstretch{1.5}
$0\leq a_1\leq a_2\leq\cdots\leq a_n$, and\\
$b_j\geq (j-i)(r-a_j)+ a_i+\cdots +a_{j-1}$,\\
for all $1\leq i \leq j \leq n$
 \end{varwidth} 
 \right.
 \right\},
 \setstretch{1}
\]
for all $r\in \Z.$
\end{theorem}

\begin{proof}

First suppose $r>0$. Proposition \ref{fg} implies that for any $m>1$, the set $\frac1m\Gamma_{rm} = \frac1m(\Gamma_r+\cdots+\Gamma_r)$ is already contained in the convex hull of $\Gamma_r$, and so the Newton-Okounkov body of $\Oo(r)$ is given by 
\begin{align*} \Delta(\Oo(r)) & = \text{closed convex hull}\left(\bigcup_{m\geq 1}\frac{1}{m} \cdot \Gamma_{rm}\right) \\
& = \text{closed convex hull}\left(\Gamma_r\right) 
\end{align*}
as claimed.\\

It follows from Lemma \ref{alt01} that $\Gamma_0=\Gamma_0+\Gamma_0$. The same argument therefore implies that $\Delta(\Oo)$ is the closed convex hull of $\Gamma_0$.\\

We temporarily use $\overline{\Delta}(\Oo(r))$ to denote the convex polyhedron in the statement of the theorem, before showing that it is equal to $\Delta(\Oo(r))$. \\

Suppose $r\geq 0$. Comparing the inequalities defining $\overline{\Delta}(\Oo(r))$ to Definition \ref{semigroup}, one sees that $\Gamma_r\subseteq \overline{\Delta}(\Oo(r))\cap \Z^{2n}$. Furthermore the only integer points of $\overline{\Delta}(\Oo(r))$ omitted from $\Gamma_r$ lie on the boundary of $\overline{\Delta}(\Oo(r))$, so we have
\[ \overline{\Delta}(\Oo(r))^\circ \cap \Z^{2n} \subseteq  \Gamma_r \subseteq \overline{\Delta}(\Oo(r))\cap \Z^{2n}. \]
Similarly for any $m>1$ we have
\[ \overline{\Delta}(\Oo(rm))^\circ \cap \Z^{2n} \subseteq  \Gamma_{rm} \subseteq \overline{\Delta}(\Oo(rm))\cap \Z^{2n}. \]
Since the inequalities defining $\overline{\Delta}(\Oo(r))$ are homogeneous in $r$, we may divide by $m$ to obtain
\[ \overline{\Delta}(\Oo(r))^\circ \cap \frac1m\Z^{2n} \subseteq  \frac1m\Gamma_{rm} \subseteq \overline{\Delta}(\Oo(r))\cap \frac1m\Z^{2n}. \]
This holds for all $m>1$, so $\Delta(\Oo(r))$ is a closed convex subset of $\overline{\Delta}(\Oo(r))$ that contains all of its interior rational points. We conclude that $\Delta(\Oo(r)) = \overline{\Delta}(\Oo(r))$ for all $r\geq 0$.\\

The same argument given after Definition \ref{semigroup} for $\Gamma_r$ shows that for all $r\in \Z$, $\overline{\Delta}(\Oo(r))$ has the alternate description,

\[\setstretch{1.5}
\overline{\Delta}(\Oo(r)) =  \left\{ \begin{varwidth}{10 in} 
$(a_1,\dots,a_n,$\\
 $b_1,\dots,b_n)\in \R^{2n}$
 \end{varwidth} \hspace{1ex}\left|\hspace{1ex} 
 \begin{varwidth}{10 in} \setstretch{1.5}
$0\leq a_1\leq a_2\leq\cdots\leq a_n$, and\\
$b_j \geq  \sum_{\substack{i=1,\dots,j-1 \\ a_j-a_i<r}} (r-a_j+a_i)$\\
for all $j$ 
 \end{varwidth} 
 \right.
 \right\}.
 \]
When $r\leq 0$, the condition $a_j-a_i<r$ never holds, as $a_i\leq a_j$ by the first inequalities. Therefore for $r\leq 0$, the second inequalities simply say that $b_1,\dots,b_n\geq 0$, and so $\overline{\Delta}(\Oo(r)) = \overline{\Delta}(\Oo) = \Delta(\Oo)$ is a simplicial cone.\\

It remains to check that $\Delta(\Oo(r)) = \Delta(\Oo)$ for $r<0$. For even $r<0$ the semigroup of $\Oo(r)$ is $\Gamma_0\oplus\Gamma_0\oplus\Gamma_0\oplus\cdots$, identical to that of $\Oo$. For odd $r<0$, the semigroup of $\Oo(r)$ is $\Gamma_0\oplus\Gamma_1\oplus\Gamma_0\oplus\Gamma_1\oplus\cdots$. One checks that in both cases the Newton-Okounkov body $\Delta(\Oo(r))$ is the same as $\Delta(\Oo)$, which completes the proof.
\end{proof}

The qualitative statements in Theorem \ref{noc2} are illustrated in the following figure, which is intended to represent a portion of the global Newton-Okounkov body of $(\C^2)^{[n]}$. In particular, there is homogeneity $\Delta(\Oo(r)) = r\Delta(\Oo(1))$ for integers $r>1$, and $\Delta(\Oo(r))$ degenerates to a simplicial cone for $r\leq 0$ .

\[\includegraphics[width = 8cm]{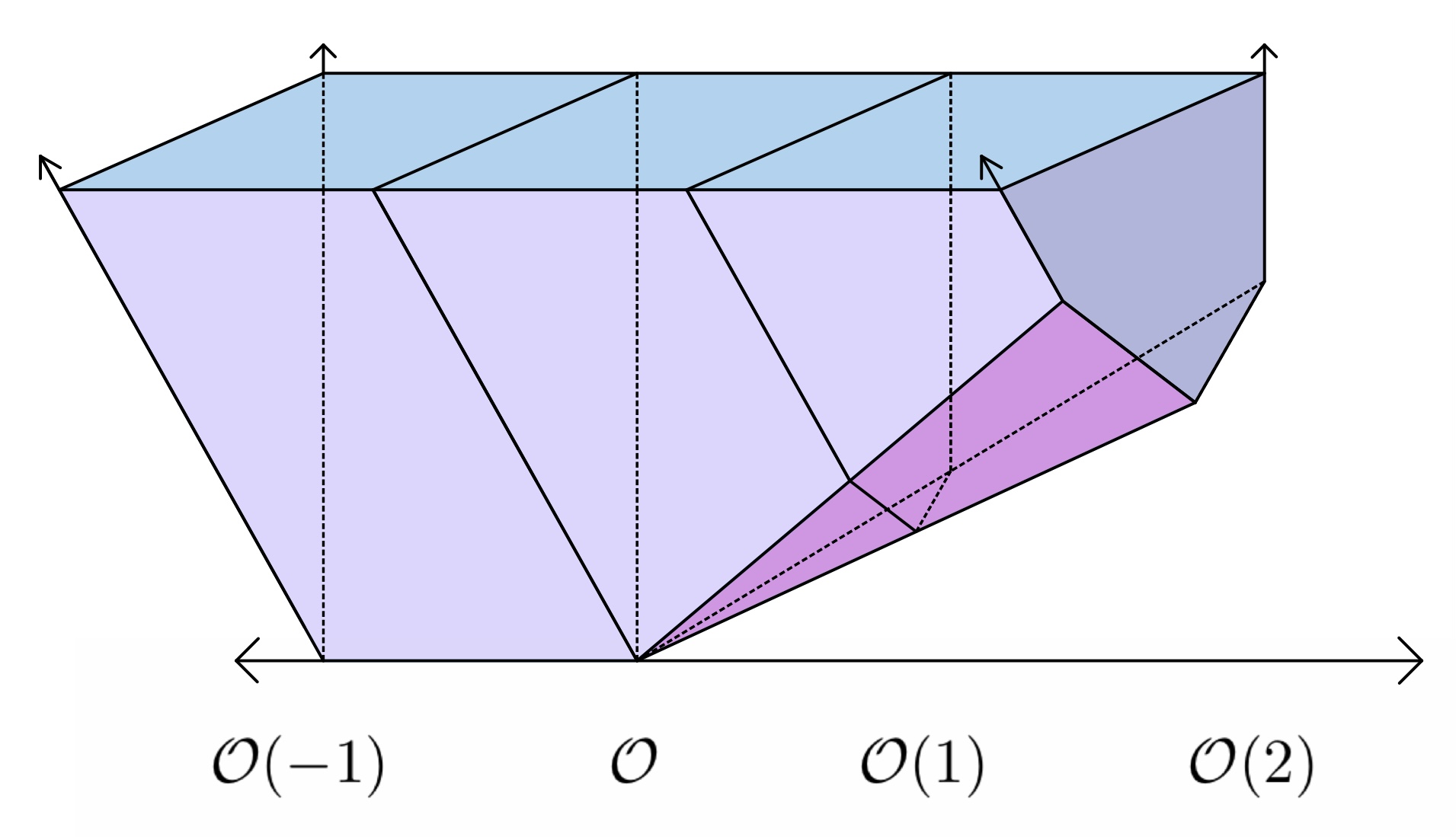}\]

The polyhedra $\Delta(\Oo(1))$ have interesting combinatorics. For example, we showed in the proof of Theorem \ref{noc2} that $\Delta(\Oo(1))$ has the alternate expression
\[\setstretch{1.5}
\Delta(\Oo(1)) =  \left\{ 
\begin{varwidth}{10 in} 
$(a_1,\dots,a_n,$\\
 $b_1,\dots,b_n)\in \R^{2n}$
 \end{varwidth}
\hspace{1ex}\left|\hspace{1ex} 
 \begin{varwidth}{10 in} \setstretch{1.5}
$0\leq a_1\leq a_2\leq\cdots\leq a_n$, and\\
$b_j \geq  \sum_{\substack{i=1,\dots,j-1 \\ a_j-a_i<1}} (1-a_j+a_i)$\\
for all $j$ 
 \end{varwidth} 
 \right.
 \right\}.
 \]
From this expression, one can show that a point $(a_1,\dots,b_n)\in \Delta(\Oo(1))$ lies on a bounded face of $\Delta(\Oo(1))$ precisely when 
\begin{itemize}
\item $a_1 =0$,
\item $a_{j+1}\leq a_j+1$ for all $j=1,\dots,n-1$, and
\item $b_j =  \sum_{\substack{i=1,\dots,j-1 \\ a_j-a_i<1}} (r-a_j+a_i)$ for all $j=1,\dots,n$.
\end{itemize}
The points on the bounded faces of $\Delta(\Oo(1))$ are therefore determined by the values $a_2,\dots,a_n$, where each $a_j$ ranges from $a_{j-1}$ to $a_{j-1}+1$. In other words, the points on the bounded faces of $\Delta(\Oo_{(\C^2)^{[n]}}(1))$ are parametrized by an $(n-1)$-cube, with coordinates given by $a_2-a_1,\dots,a_n-a_{n-1}.$ The combinatorics of these bounded faces corresponds to a polyhedral subdivision of the $(n-1)$-cube into regions depending on which of the pairs $1\leq i<j\leq n$ have $a_j-a_i<1$. The number of top-dimensional cells in this polyhedral subdivision of the $(n-1)$-cube, and therefore the number of top-dimensional bounded faces of $\Delta_{(\C^2)^{[n]}}(\Oo(1))$, is the Catalan number $C_{n-1}$.\\

\subsection{The Moment Polytope and Duistermaat-Heckman measure for $(\C^2)^{[n]}$}\label{DH}

In this section we show how the unbounded polyhedron $\Delta(\Oo(1))\subseteq \R^{2n}$ encodes asymptotic information about the sections of $\Oo(r)$ equivariantly. \\

The spaces $A^r\simeq H^0((\C^2)^{[n]},\Oo(r))$ decompose into graded pieces
\[ A^r = \bigoplus_{(p,q)\in \Z^2_{\geq 0}} A^r_{(p,q)} ,\]
with respect to the $\Z^2$-grading defined at the end of Section \ref{algsection}. The main observation of this section is that the valuation $\nu$ is compatible with this grading. In other words, each valuation $v\in \Gamma_r$ obtained from a polynomial $f\in A^r$ is also obtained from a homogeneous polynomial $f\in A^r_{(p,q)}$ for some $(p,q)\in \Z^2_{\geq 0}$, and the degree $(p,q)$ is uniquely determined by the vector $v\in \Gamma_r.$ This fact was already used in the proof of Corollary \ref{int}, and follows from the proof of Proposition \ref{sg}.

\begin{corollary}
For all $(p,q)\in \Z^2_{\geq0}$, the set of valuations $\nu(f)$ obtained by nonzero homogeneous polynomials $f\in A^r_{(p,q)}$ is equal to $\Gamma_r^{(p,q)}$, where
\[ \Gamma_r^{(p,q)}= \left\{ (p_1,\dots,p_n,q_1,\dots,q_n)\in \Gamma_r \hspace{1ex} \left| \hspace{1ex} 
 \begin{varwidth}{10 in} \setstretch{1.5}
$p_1+\cdots+p_n = p$\\
$q_1+\cdots+q_n = q$
 \end{varwidth} \right.  \right\} \]
The dimension of the $(p,q)$-graded piece of $H^0((\C^2)^{[n]},\Oo(r))$ is therefore given by $\#\Gamma_r^{(p,q)}$. \qed
\end{corollary}

Following \cite{BP}, the dimensions of the graded pieces of $H^0((\C^2)^{[n]},\Oo(r))$ can be encoded as Dirac measures,

\[ \sum_{(p,q)\in \Z^2_{\geq0}}\#\Gamma_r^{(p,q)}\cdot\delta_{(p,q)}. \]

The Duistermaat-Heckman measure of $((\C^2)^{[n]},\Oo(1),T)$, which we denote by $\DHeck(n)$, is the weak limit of the rescaled Dirac measures, 
\[ \DHeck(n) := \lim_{r\to\infty} \sum_{(p,q)\in \Z^2_{\geq0}} \frac{\#\Gamma_r^{(p,q)}}{r^{2n}}\delta_{(p/r,q/r)}, \]
considered as a measure on $\R^2$. The measure $\DHeck(n)$ is equal to a piecewise-polynomial function times the usual Lebesgue measure on $\R^2$. The Newton-Okounkov body $\Delta(\Oo(1))$ encodes this measure as follows. Let $\pi:\R^{2n}\to \R^2$ denote the linear projection $(a_1,\dots,a_n,b_1,\dots,b_n)\mapsto (a_1+\cdots+a_n,b_1+\cdots+b_n)$.

\begin{corollary}
The Duistermaat-Heckman measure $\DHeck(n)$ is equal to the pushforward $\pi_*(d\mu|_{\Delta(\Oo(1))}),$ where $d\mu|_{\Delta(\Oo(1))}$ is the Lebesgue measure on $\R^{2n}$ restricted to $\Delta(\Oo(1)).$ 
\end{corollary}

Okounkov's original construction was in a similar context \cite{O2}, and results of this form hold much more generally (c.f. Theorem 1.7 of \cite{KK2}). Since we have explicit descriptions of the sets $\Gamma_r^{(p,q)}$ and the Newton-Okounkov body, the proof is straightforward, so we include it for completeness.

\begin{proof}
As in the proof of Theorem \ref{noc2}, we have
\[ \Delta(\Oo(r))^\circ \cap \Z^{2n} \subseteq \Gamma_r \subseteq \Delta(\Oo(r))\cap \Z^{2n}, \]
and by homogeneity 
\[ \Delta(\Oo(1))^\circ \cap \frac1r\Z^{2n} \subseteq \frac1r\Gamma_r \subseteq \Delta(\Oo(1))\cap \frac1r\Z^{2n}. \]
This implies the weak convergence,
\[ \sum_{v\in \Gamma_r} \frac{1}{r^{2n}}\delta_{v/r} = \sum_{v\in \frac1r \Gamma_r} \frac{1}{r^{2n}}\delta_{v} \xrightarrow{r\to\infty} d\mu|_{\Delta(\Oo(1))}. \]
Pushing forward by $\pi$ we obtain the desired result.
\end{proof}

This result is depicted in the following diagram. The point masses in the measure obtained from $\Oo(r)$ approximately count the number of $1/r$-integer points in the fibers of $\Delta(\Oo(1))$. In the limit $r\to\infty$, the density function is given by the volumes of the fibers.

\[\begin{tikzpicture}
\draw[thick,<->] (-1.5,-.2) -- (4,-.2);
\filldraw[black] (4,-.2) circle (0pt) node[anchor=west]{$\R^2$};
\draw [fill=blue, fill opacity=0.2,thick]
       (3,1) -- (0,1) -- (-0.5,1.5) -- (-0.5,2.2)--(1.5,2.7+1.5);
\draw [thick,->]
       (3,1) --(3.1,1);
\draw [thick,->]
       (1.5,2.7+1.5) -- (1.6,2.7+1.6);
\filldraw[black] (0.7,2.2) circle (0pt) node[anchor=west]{$\Delta(\mathcal{O}(1))$};
\draw [->,thick]
    (1.3,0.7) -- (1.3,0.1);
\filldraw[black] (1.3,0.4) circle (0pt) node[anchor=west]{$\pi$};
\filldraw[black] (0.41,-0.2) circle (1.5pt) node[anchor=north]{$(p,q)$};
\draw (0.41,-0.2) -- (0.41,3.1);
\filldraw[black] (0.41,1) circle (1.5pt);
\filldraw[black] (0.41,1.3) circle (1.5pt);
\filldraw[black] (0.41,1.6) circle (1.5pt);
\filldraw[black] (0.41,1.9) circle (1.5pt);
\filldraw[black] (0.41,2.2) circle (1.5pt);
\filldraw[black] (0.41,2.5) circle (1.5pt);
\filldraw[black] (0.41,2.8) circle (1.5pt);
\filldraw[fill = white] (0.41,3.1) circle (1.5pt);
\end{tikzpicture} \]

\section{The Hilbert Schemes of Points on Projective Toric Surfaces}\label{projsection}

Let $X$ be a smooth, projective, toric surface. As recalled in Section \ref{Hilbbg}, any divisor on the Hilbert scheme $X^{[n]}$ is linearly equivalent to $D_n+rE$ for some divisor $D\in \Pic(X)$ and $r\in \Z$. We identify the sections of $\Oo(D_n+rE)$ with a subset of the sections of $\Oo_{(\C^2)^{[n]}}(r)$ that satisfy a term restriction coming from the Newton polytope of $D$ (Proposition \ref{projsections}). This identification allows us to study the Newton-Okounkov body of $X^{[n]}$ using our results on $(\C^2)^{[n]}$.\\

The main difference is that for projective surfaces we only obtain upper bounds for the graded semigroups (Proposition \ref{projsgub}) and Newton-Okounkov bodies (Theorem \ref{projub}). The extra difficulty over the $\C^2$ case is that the semigroups in the projective case are not generated in degree one. \\

For most toric surfaces $X$, these upper bounds are not sharp. We conjecture, however, that for the Hilbert schemes of points on $\P^2,$ $\P^1\times \P^1$, and Hirzebruch surfaces, the convex bodies appearing in Theorem \ref{projub} are the exact Newton-Okounkov bodies (Conjecture \ref{sharp}). This conjecture is supported by calculations on these surfaces for small $n$. In fact, the inequalities appearing in Theorem \ref{projub} (and Theorem \ref{noc2}) were originally found based on explicit computations on $(\P^2)^{[n]}$ for $n=2,3,4$. To illustrate these techniques, we verify Conjecture \ref{sharp} for the Hilbert scheme of $4$ points on $\P^2$, computing the global Newton-Okounkov body $\Delta((\P^2)^{[4]})$.\\

In Section \ref{effsection} we apply these results to study the effective cone of $X^{[n]}$.

\subsection{Global Sections of Line Bundles on $X^{[n]}$}

Let $X$ be a smooth, projective, toric surface and $D$ a $T$-invariant divisor on $X$. In Section \ref{toric} we recalled the definition of the Newton polytope of $D$, and made the identification $P_D\subseteq \R^2$. The polygon $P_D$ controls the sections of $\Oo(D)$ via the formula
\[ H^0(X,\Oo(D)) \simeq \bigoplus_{(p,q)\in P_D\cap \Z^2} \C\cdot x^py^q. \]
The identification $P_D\subseteq \R^2$ depends on a choice of coordinates $\C^2\simeq U_\sigma\subseteq X$. Fixing such an isomorphism $\C^2 \simeq U_\sigma \hookrightarrow X$ yields open embeddings $(\C^2)^{[n]}\simeq U_\sigma^{[n]} \hookrightarrow X^{[n]}$ for all $n$, by which we can identify sections of line bundles on $X^{[n]}$ with their restrictions to $(\C^2)^{[n]}$. In Section \ref{CoordRingc2Section} we proved that 
\[ H^0((\C^2)^{[n]},\Oo(r))\simeq A^r \]
for all $r\in \Z$ (with $A^r$ defined as in Section \ref{noc2section} for $r<0$). The sections of $\Oo(D_n+rE) \in \Pic(X^{[n]})$ can be expressed as follows:

\begin{definition}\label{projsgdef}
For a divisor $D\in \Pic(X)$ on a smooth, projective, toric surface $X$ with Newton polygon $P_D\subseteq \R^2$ and $r\in \Z$, define
\[ A(D_n+rE) = A^r \cap \bigoplus_{(p_i,q_i)\in P_D\cap \Z^2}\C \cdot x_1^{p_1}y_1^{q_1}\cdots x_n^{p_n}y_n^{q_n}. \]
\end{definition}

\begin{proposition}\label{projsections}
Let $D\in \Pic(X)$ be a divisor on a smooth, projective, toric surface $X$ with Newton polygon $P_D\subseteq \R^2$. Then for any $r\in \Z$, we have
\[ H^0\left(X^{[n]},\Oo\left(D_n+rE\right)\right) \simeq A(D_n+rE). \]
\end{proposition}

As with the corresponding result for $\C^2$ (Corollary \ref{int}), this description may be known to experts, but we have been unable to locate a reference.

\begin{proof}
For $r=0,1$, these isomorphisms are the well-known identifications 
\[ H^0\left(X^{[n]},\Oo(D_n)\right)\simeq \Sym^n H^0(X,\Oo(D)),\] 
and 
\[H^0\left(X^{[n]},\Oo\left(D_n+E\right)\right)\simeq \bigwedge^n H^0(X,\Oo(D)).\]

See, for example, the proof of Lemma 5.1 in \cite{EGL}.\\

Since $E=-\frac12B$, there is a section $s\in H^0(X^{[n]},\Oo(-2E))$ defining the divisor $B$. Multiplication by $s$ defines maps
\[ H^0\left(X^{[n]},\Oo\left(D_n+(r+2)E\right)\right) \to H^0\left(X^{[n]},\Oo\left(D_n+ rE\right)\right) \]
for all integers $r$. The same argument as in the proof of Lemma \ref{neg} (and \cite{BC} Proposition 3.2) shows that these maps are isomorphisms for all $r<0$. Composing these isomorphisms from either $\Oo(D_n)$ or $\Oo(D_n+E)$ proves the claim in the case $r<0$, by the definition of $A^r$ for $r<0$. \\

For $r>1$, we can again compose these embeddings to identify the global sections of $\Oo(D_n+rE)$ with the subspace of sections of $\Oo(D_n)$ or $\Oo(D_n+E)$ that vanish along $B$ to the appropriate order. To compute this order of vanishing, we may first restrict the section to $(\C^2)^{[n]}$, where the section is identified with a polynomial in $A^0$ or $A^1$, then compute the order of vanishing along $B|_{(\C^2)^{[n]}}$. Corollary \ref{int} can be interpretted as saying that the filtrations of $A^0$ and $A^1$ according to order of vanishing of the corresponding sections along $B$ are given by
\[ A^0\supseteq A^2 \supseteq A^4 \supseteq \cdots, \]
and
\[ A^1\supseteq A^3 \supseteq A^5 \supseteq \cdots. \]
This, along with the known cases $r=0,1$, establishes the identification for $r>1$.
\end{proof}

This result, combined with Corollary \ref{int}, implies a pleasant identification of the sections of $\Oo(D_n+rE)$ for $r\geq 0$ with the set of polynomials in $\C[\mathbf{x,y}] = \C[x_1,\dots,x_n,y_1,\dots,y_n]$ that:
\begin{enumerate}
    \item Are symmetric (when $r$ is even) or alternating (when $r$ is odd),
    \item Are contained in the ideal $J^r = \bigcap_{i<j}(x_i-x_j,y_i-y_j)^r$, and
    \item Have Newton polytope contained in $P_D$, when considered as a polynomial in any one of the pairs of variables $(x_i,y_i)$.
\end{enumerate}

One can translate constructions of divisors on $X^{[n]}$ into the language of these (anti-)symmetric polynomials. The following examples are enlightening, although we do not need them in this paper.

\begin{example}
Let $C\subseteq X$ be an irreducible curve representing the divisor class $D$. Restricting to the affine open $\C^2\simeq U_\sigma \subseteq X$, $C$ is defined as the vanishing locus $f(x,y)=0$ for some polynomial $f$ whose Newton polygon is contained in $P_D$. The divisor $D_n$ is represented by the locus of length $n$ subschemes of $X$ whose support meets the curve $C$. This representative corresponds to the polynomial \[ \prod_{i=1}^n f(x_i,y_i)\in A(D_n) \]
\end{example}

\begin{example}
Fix a divisor $D$ on $X$, and let $f_1,\dots,f_n$ be linearly independent sections of $\Oo(D)$. The span of $f_1,\dots,f_n$ corresponds to a linear system of curves in $X$, and the divisor $D_n+E$ is represented the locus of length $n$ subschemes $Z\subseteq X$ such that there exists a curve $C_Z$ in this linear system that contains $Z$ as a closed subscheme. This representative corresponds to the polynomial
\[ \det(f_i(x_j,y_j))_{ij} \in  A(D_n+E). \]
\end{example}

\begin{example}\label{vb}
Let $F$ be a vector bundle on $X$ of rank $r$ with  $s_1,\dots,s_{rn}$ general sections of $F$. The restricted sections $s_i|_{\C^2}$ can be represented as $r$-tuples of polynomials, 
\[ s_i|_{\C^2} = (f_{i,1} \cdots f_{i,r})^T. \]
Consider the polynomial $d(\mathbf{x,y})$ defined by
\[ \det\begin{pmatrix}
f_{1,1}(x_1,y_1) & \cdots & f_{1,r}(x_1,y_1) &\cdots & f_{1,1}(x_n,y_n) & \cdots & f_{1,r}(x_n,y_n) \\
\vdots \\
f_{rn,1}(x_1,y_1) & \cdots & f_{rn,r}(x_1,y_1) &\cdots & f_{rn,1}(x_n,y_n) & \cdots & f_{rn,r}(x_n,y_n) \\
\end{pmatrix}. \]
If the polynomial $d$ is not identically zero, then $F$ is said to satisfy interpolation for $n$ points, and $d$ corresponds to a divisor of class $c_1(F)_n+rE$. In this case one can check directly that $d\in A(c_1(F)_n+rE)$. Indeed, exchanging any pair of variables $(x_i,y_i)$ with $(x_j,y_j)$ in the matrix above swaps $r$ columns, so $d$ is multiplied by a factor of $(-1)^r.$ We also have
\[ d \in \bigcap_{i<j}(x_i-x_j,y_i-y_j)^r, \]
since setting $x_i=x_j$ and $y_i=y_j$ makes $r$ pairs of columns repeat in the matrix above. Finally, in each pair of variables $(x_i,y_i)$, $d$ is expressed as a linear combination of determinants
\[ d' =  \det\begin{pmatrix}
f'_{1,1}(x_i,y_i) & \cdots & f'_{1,r}(x_i,y_i)\\
\vdots \\
f'_{r,1}(x_i,y_i) & \cdots & f'_{r,r}(x_i,y_i)
\end{pmatrix} \]
where the vectors $(f'_{i,1} \cdots f'_{i,r})^T$ for $i=1,\dots, r$ represent $r$ general sections of $F$. Each $d'$ represents a general section of $\bigwedge^r F$, and therefore has Newton polytope contained in $P_{c_1(\bigwedge^r F)} = P_{c_1(F)}$. This establishes the Newton polytope restriction for $d$ as well, which shows that $d\in A(c_1(F)_n+rE)$ as desired.
\end{example}

\subsection{The Graded Semigroup of $X^{[n]}$}

Let $D$ be a $T$-invariant divisor on a smooth, projective, toric variety $X$ with Newton polygon \[ P_D = \left\{ (a,b)\in \R^2 \hspace{1ex}\left|\hspace{1ex} 
 \begin{varwidth}{10 in} \setstretch{1.5}
$0\leq a \leq c$, and\\
$\ell(a) \leq b \leq u(a)$
 \end{varwidth} 
 \right.
 \right\}\]
for some constant $c$, and piecewise linear functions $\ell$ and $u$. By analogy with the sets of valuations $\Gamma_r$ from $A^r$, we define a candidate set of valuations from $A(D_n+rE)$.
\begin{definition}
With $D$ and $P_D$ as above, and $r\geq0$, let $\Gamma(D_n+rE)$ be
\[ \setstretch{1.5}
\left\{ \begin{varwidth}{10 in} \setstretch{1.5}
$(p_1,\dots,p_n,$\\
 $q_1,\dots,q_n)\in \Z^{2n}_{\geq 0}$
 \end{varwidth}  \hspace{1ex}\left|\hspace{1ex} 
 \begin{varwidth}{10 in} \setstretch{1.5}
$0 \leq p_1\leq p_2\leq\cdots\leq p_n\leq c$,\\
if $p_{j} = p_{j+1}$ then $q_{j+1}\geq q_j+r$, \\
$q_j\geq \ell(p_j) +(j-i)(r-p_j) + p_i+\cdots +p_{j-1}$,\\
$q_j\leq u(p_j) - (k-j)(r+p_j) + p_{j+1}+\cdots+p_k$,\\
for all $1\leq i \leq j \leq k\leq n$
 \end{varwidth} 
 \right.
 \right\}.
 \setstretch{1}
\]
For even $r<0$ define $\Gamma(D_n+rE) = \Gamma(D_n)$, and for odd $r<0$ define $\Gamma(D_n+rE) = \Gamma(D_n+E)$.
\end{definition}

\begin{proposition}\label{projsgub}
For all $r\in \Z$ we have 
\[ \Gamma(D_n+rE) \supseteq \{ \nu(f) \hspace{1ex}|\hspace{1ex}f\in A(D_n+rE)\setminus\{0\} \}, \]
and for $r\leq 1$ we have
\[ \Gamma(D_n+rE) = \{ \nu(f) \hspace{1ex}|\hspace{1ex}f\in A(D_n+rE)\setminus\{0\} \}. \]
\end{proposition}

\begin{proof}
Suppose $r\geq 0 $, and fix a nonzero polynomial $f\in A(D_n+rE)$ with $\nu(f) = (p_1,\dots,p_n,q_1,\dots,q_n)$. Since $f\in A^r$ we have $0 \leq p_1\leq p_2\leq\cdots\leq p_n$. By the definition of $A(D_n+rE)$, the Newton polytope of $f$ is contained in $(P_D)^n$, so Proposition \ref{np} implies that $P_D$ contains the points 
\[ \left( p_j, q_j -  \sum_{\substack{i=1,\dots,j-1 \\ p_j-p_i<r}} (r-p_j+p_i)\right) \text{ and } \left( p_j, q_j + \sum_{\substack{k = j+1,\dots,n \\ p_k-p_j<r}} (r-p_k+p_j)\right) \]
for all $j=1,\dots,n$. By the definition of $P_D$, this is equivalent to $p_j\leq c$ and
\[ \ell(p_j)\leq q_j -  \sum_{\substack{i=1,\dots,j-1 \\ p_j-p_i<r}} (r-p_j+p_i) \leq q_j \leq q_j + \sum_{\substack{k = j+1,\dots,n \\ p_k-p_j<r}} (r-p_k+p_j) \leq u(p_j) \]
for all $j=1,\dots,n$. The same argument given for $\Gamma_r$ after Definition \ref{semigroup} shows that the inequalities above are equivalent to the final inequalities in the definition of $\Gamma_r(D).$ That $p_{j} = p_{j+1}$ implies $q_{j+1}\geq q_j+r$ follows from Proposition \ref{np}, which completes the proof that $\nu(f)\in \Gamma_r$.\\

By the same argument as in the proof of Lemma \ref{alt01}, $\Gamma(D_n)$ and $\Gamma(D_n+E)$ can be described as
\[ \Gamma(D_n) = \left\{ \begin{varwidth}{10 in} \setstretch{1.5}
$(p_1,\dots,p_n,$\\
 $q_1,\dots,q_n)\in \Z^{2n}$
 \end{varwidth}  \hspace{1ex}\left|\hspace{1ex}
 \begin{varwidth}{10 in} \setstretch{1.5}
$(p_1,q_1)\leq \cdots \leq (p_n,q_n)$\\
in lex, and
$(p_j,q_j)\in P_D$ for all $j$
 \end{varwidth} 
 \right.
 \right\} , \]
and 
\[ \Gamma(D_n+E) = \left\{ \begin{varwidth}{10 in} \setstretch{1.5}
$(p_1,\dots,p_n,$\\
 $q_1,\dots,q_n)\in \Z^{2n}$
 \end{varwidth} \hspace{1ex}\left|\hspace{1ex}
 \begin{varwidth}{10 in} \setstretch{1.5}
$(p_1,q_1)< \cdots < (p_n,q_n)$\\
in lex, and
$(p_j,q_j)\in P_D$ for all $j$
 \end{varwidth} 
 \right.
 \right\} . \]
These are precisely the sets of valuations obtained by the bases $m_{(\mathbf{p,q})}(\mathbf{x,y})$ and $d_{(\mathbf{p,q})}(\mathbf{x,y})$ of Section \ref{01} with $(p_j,q_j)\in P_D$ for all $j$. This implies the equality $\Gamma(D_n+rE) = \{ \nu(f) \hspace{1ex}|\hspace{1ex}f\in A(D_n+rE)\setminus\{0\}$ for $r=0,1$. By the definitions of $\Gamma(D_n+rE)$ and $A(D_n+rE)$ for $r<0$, we obtain the same equality for all $r<0$ as well.
\end{proof}

\subsection{The Newton-Okounkov Body of $X^{[n]}$}

By Proposition \ref{projsections}, the Newton-Okounkov body of $D_n+rE\in \Pic(X^{[n]})$ can be defined as
\[ \Delta(D_n+rE) = \text{closed convex hull}\left(\bigcup_{m>0}\frac{1}{m} \cdot \left\{ \nu(f) \hspace{1ex}\big|\hspace{1ex} f\in A(mD_n+mrE)\setminus\{0\} \right\}\right). \]

\begin{theorem}\label{projub}
For all $D_n+rE\in \Pic(X^{[n]})$, the Newton-Okounkov body $\Delta\left(D_n+rE\right)$ is contained in the set $\overline{\Delta}(D_n+rE)$, which is defined by the inequalities
\[ \setstretch{1.5}
\left\{ \begin{varwidth}{10 in} 
$(a_1,\dots,a_n,$\\
 $b_1,\dots,b_n)\in \R^{2n}$
 \end{varwidth} \hspace{1ex}\left|\hspace{1ex} 
 \begin{varwidth}{10 in} \setstretch{1.5}
$0\leq a_1\leq a_2\leq\cdots\leq a_n \leq c$, and\\
$b_j\geq \ell(a_j) +(j-i)(r-a_j) + a_i+\cdots +a_{j-1}$,\\
$b_j\leq u(a_j) - (k-j)(r+a_j) + a_{j+1}+\cdots+a_k$,\\
for all $1\leq i \leq j \leq k\leq n$
 \end{varwidth} 
 \right.
 \right\}.
 \setstretch{1}
\]
For $r\leq 0$, we have $\overline{\Delta}(D_n+rE) = \Delta(D_n+rE) = \Delta(D_n)$, and the simpler description,
\[ \Delta(D_n) = \setstretch{1.5}
\left\{ (a_1,\dots,a_n,b_1,\dots,b_n)\in \R^{2n} \hspace{1ex}\left|\hspace{1ex} 
 \begin{varwidth}{10 in} \setstretch{1.5}
$a_1\leq a_2\leq\cdots\leq a_n $, and\\
$(a_j,b_j)\in P_D$ for all $j = 1,\dots,n$
 \end{varwidth} 
 \right.
 \right\}.
 \setstretch{1}
\]
\end{theorem}

\begin{proof}
It follows from the definitions that for all $D$ and $r$, we have $\Gamma(D_n+rE)\subseteq \overline\Delta(D_n+rE)$. By the homogeneity of the Newton polygons $P_{mD}= mP_D$, the inequalities defining $\overline{\Delta}(D_n+rE)$ are homogeneous in the input, so for any $m>0$ we have
\[\frac1m\Gamma(mD_n+mrE)\subseteq \frac1m\overline{\Delta}\left(mD_n+mrE\right) = \overline{\Delta}\left(D_n+rE\right) \]

By Proposition \ref{projsgub}, $\Gamma(mD_n+mrE)$ contains all the valuation vectors from polynomials in $A(mD_n+mrE)$, so we have
\[ \Delta(D_n+rE) \subseteq \text{closed convex hull} \left( \bigcup_{m>0} \frac 1m \Gamma(mD_n+mrE) \right) \subseteq \overline{\Delta}(D_n+rE). \]

To establish the equality $\overline{\Delta}(D_n) = \Delta(D_n)$ in the case $r=0$, we note that for any $m>0$, $\frac1m\Gamma(mD_n)$ contains all of the interior $\frac1m$-integer points of $\overline{\Delta}(D_n)$. By Proposition \ref{projsections}, $\Gamma(mD_n)$ is precisely the set of valuations of polynomials in $A(mD_n)$. This implies that the Newton-Okounkov body $\Delta(D_n)$ contains all the interior rational points of $\overline{\Delta}(D_n)$, which establishes the remaining inclusion $\Delta(D_n)\supseteq \overline{\Delta}(D_n).$\\

Finally, we check the case $r<0$. By the homogeneity of Newton-Okounkov bodies, we may assume $r<0$ is even. In this case, we have by definition $A(mD_n+mrE) = A(mD_n)$ for all $m>1$, so $\Delta(D_n+rE) = \Delta(D_n) = \overline{\Delta}(D_n).$ It remains to check that $\overline{\Delta}(D_n+rE)= \overline{\Delta}(D_n).$ This can be seen from the alternate expression for $\overline{\Delta}(D_n+rE),$ which holds for all $r$,
\[ \setstretch{1.5}
\left\{ (a_1,\dots,a_n,b_1,\dots,b_n)\in \R^{2n} \hspace{1ex}\left|\hspace{1ex} 
 \begin{varwidth}{10 in} \setstretch{1.5}
$0\leq a_1\leq a_2\leq\cdots\leq a_n \leq c$, and\\
$b_j\geq \ell(a_j) + \sum_{\substack{i=1,\dots,j-1 \\ a_j-a_i<r}} (r-a_j+a_i)$,\\
$b_j\leq u(a_j) - \sum_{\substack{k = j+1,\dots,n \\ a_k-a_j<r}} (r-a_k+a_j)$,\\
for all $j=1,\dots,n$.
 \end{varwidth} 
 \right.
 \right\}.
 \setstretch{1}
\]
This alternate description can be established by the same argument given for $\Gamma_r$ after Definition \ref{semigroup}. With this description one can see that whenever $r\leq 0$, the conditions $a_j-a_i<r$ and $a_k-a_j<r$ never hold, since $i<j<k$ implies $a_i\leq a_j\leq a_k$. Thus for all $r\leq 0$ these final conditions reduce to the inequalities $\ell(a_j)\leq b_j\leq u(a_j)$ for all $j=1,\dots,n.$ This establishes the alternate description of $\Delta(D_n)$ in the statement of the theorem, and shows that $\overline{\Delta}(D_n) = \overline{\Delta}(D_n+rE)$ whenever $r<0$, completing the proof.
\end{proof}

\begin{remark}
By a similar argument given in the case $r=0$ above, one can show that $\overline{\Delta}(D_n+rE)$ is equal to
\[ \text{closed convex hull}\left(\bigcup_{m>0}\frac{1}{m} \cdot \Gamma(mD_n+mrE)\right) \]
for all divisors $D_n+rE\in \Pic(X^{[n]})$. 
\end{remark}

In the case $r=0$, the divisors $D_n\in \Pic(X^{[n]})$ are obtained by pulling back from $X^{(n)}$ via the Hilbert-Chow morphism. The convex sets $\Delta(D_n)$ can therefore be interpreted as Newton-Okounkov bodies on $X^{(n)}$. By the previous theorem, $\Delta(D_n)$ is identified with the set of $n$-tuples $(a_1,b_1),\dots,(a_n,b_n)\in P_D$ such that $a_1\leq \cdots\leq a_n.$ From this description one sees that the Euclidean volume of $\Delta(D_n)\subseteq \R^{2n}$ is given by
\[ \vol_{\R^{2n}}(D_n) = \frac{1}{n!}(\vol_{\R^2}(P_D))^n. \]
By Theorem \ref{vol} on volumes of Newton-Okounkov bodies, this gives
\[ \frac{\vol_{X^{[n]}}(D_n)}{(2n)!} = \frac{1}{n!} \left( \frac{\vol_X(D)}{2!} \right)^n. \]
The Newton-Okounkov body $\Delta(D_n)$ therefore gives a convex geometric interpretation for this known formula for $\vol(D_n)$. This formula can also be proved directly from the identity
\[ h^0(X^{[n]},\Oo(D_n)) = {h^0(X,\Oo(D))+n-1 \choose n}. \]

For $r\leq 0$, the Newton-Okounkov bodies $\Delta(D_n+rE)$ are constant equal to $\Delta(D_n)$. This leaves the case $r>0$, where we have only obtained an upper bound on the Newton-Okounkov body. Effective divisors in this remaining region are difficult to describe, so it is unsurprising that these Newton-Okounkov bodies are more difficult to compute. This is discussed further in the remaining sections.

\subsection{Examples and Conjectures}\label{examples}

We continue to use $\overline\Delta(D_n+rE)$ to denote the convex body appearing in Theorem \ref{projub}. The containment 
\[ \Delta\left(D_n+rE\right)\subseteq \overline\Delta\left(D_n+rE\right) \]
of Theorem \ref{projub} is strict for most toric surfaces $X$. However, we propose:

\begin{conjecture}\label{sharp}
If the surface $X$ is $\P^2, \P^1\times \P^1$, or a Hirzebruch surface, then $\overline\Delta(D_n+rE) = \Delta(D_n+rE)$  for all divisors $D_n+rE\in \Pic(X^{[n]})$.
\end{conjecture}

For $\P^2$ and $\P^1\times \P^1$, the choice of coordinates does not matter. For Hirzebruch surfaces however, one must choose coordinates so that the Newton polygons are oriented as in the pictures at the top of the table on the final page. This asymmetry apparently comes from our choice of valuation.

\begin{example}
We check that Conjecture \ref{sharp} holds for $(\P^2)^{[4]}$. Denote the image of the class of a line in $\Pic((\P^2)^{[4]})$ by $H$. In this case, the Newton polygon of $dH$ is the right triangle
\[ P_{dH} = \left\{ (a,b)\in \R^2 \hspace{1ex}\left|\hspace{1ex} 
 \begin{varwidth}{10 in} \setstretch{1.5}
$a\geq 0$, and\\
$0 \leq b \leq d-a$
 \end{varwidth} 
 \right.
 \right\}, \]
so the Newton-Okounkov body $\Delta(dH+rE)$ is contained in the convex set $\overline{\Delta}(dH+rE)$, defined by
\[ \setstretch{1.5}
\left\{ \begin{varwidth}{10 in} 
$(a_1,\dots,a_4,$\\
 $b_1,\dots,b_4)\in \R^{8}$
 \end{varwidth} \hspace{1ex}\left|\hspace{1ex} 
 \begin{varwidth}{10 in} \setstretch{1.5}
$0\leq a_1\leq a_2\leq a_3 \leq a_4$, and\\
$b_j\geq 0+(j-i)(r-a_j) + a_i+\cdots +a_{j-1}$,\\
$b_j\leq (d-a_j) - (k-j)(r+a_j) + a_{j+1}+\cdots+a_k$,\\
for all $1\leq i \leq j \leq k\leq 4$
 \end{varwidth} 
 \right.
 \right\}.
 \setstretch{1}
\]

One can check by computer that $\overline{\Delta}(dH+rE)$ is nonempty exactly when $dH+rE$ is in the convex cone spanned by $-E$ and $3H+2E$, i.e. when $dH+rE$ is effective \cite{ABCH}. The effective cone has a chamber decomposition such that in each chamber the convex bodies $\overline{\Delta}(D)$ vary linearly. For example, divisors in the cone spanned by $(3H+E)$ and $H$ can be written as $x(3H+E)+yH$ for $x,y\geq0$, and one has
\[ \overline{\Delta}(x(3H+E)+yH) = x\overline{\Delta}(3H+E)+ y\overline{\Delta}(H) \]
for all $x,y\geq 0$. The complete decomposition of $\Eff((\P^2)^{[4]})$ in this way is depicted below.

\[ \begin{tikzpicture}
\draw[step=1cm,gray,very thin] (-3,0) grid (3,5);
\draw[thick,->] (0,0) -- (-3.5,0);
\draw[thick,->] (0,0) -- (0,5.5);
\draw[thick,->] (0,0) -- (1/2*5.5,5.5);
\draw[thick,->] (0,0) -- (2/6*5.5,5.5);
\draw[thick,->] (0,0) -- (3.5,5.25);
\filldraw[black] (1,0) circle (0pt) node[anchor=north]{$E$};
\filldraw[black] (2,0) circle (0pt) node[anchor=north]{$2E$};
\filldraw[black] (-1,0) circle (0pt) node[anchor=north]{$-E$};
\filldraw[black] (-2,0) circle (0pt) node[anchor=north]{$-2E$};
\filldraw[black] (0,0) circle (0pt) node[anchor=north]{$0$};
\filldraw[black] (0,1) circle (0pt) node[anchor=east]{$H$};
\filldraw[black] (0,2) circle (0pt) node[anchor=east]{$2H$};
\filldraw[black] (0,3) circle (0pt) node[anchor=east]{$3H$};
\filldraw[black] (0,4) circle (0pt) node[anchor=east]{$4H$};
\filldraw[black] (2,3) circle (2pt);
\filldraw[black] (1,2) circle (2pt);
\filldraw[black] (1,3) circle (2pt);
\filldraw[black] (0,1) circle (2pt);
\filldraw[black] (-1,0) circle (2pt);
\draw [fill=blue, fill opacity=0.2,draw opacity=0]
       (0,0) -- (3,4.5) -- (3,5) -- (-3,5) -- (-3,0) -- cycle;
\end{tikzpicture} \]

By Theorem \ref{projub}, we have $\overline{\Delta}(D)=\Delta(D)$ for all divisors $D$ in the cone spanned by $-E$ and $H$. Since $\Delta(-E)= \{\vec 0\}$, linearity in this chamber says that $\Delta(tH-sE) = t\Delta(H)$ for all $t,s\geq 0$.\\

It remains to check that $\overline{\Delta}(dH+rE) = \Delta(dH+rE)$ in the case $r>0$. Consider the ample divisor $4H+E$. By Theorem \ref{vol} on volumes of Newton-Okounkov bodies, and the fact that the volume of an ample divisor is equal to its top self intersection number, we have
\begin{align*}
     \vol_{\R^8}\left(\Delta\left(4H+E\right)\right) & = \frac{1}{8!}\vol_{(\P^2)^{[4]}} \left(4H+E\right) \\ & = \frac{1}{8!}\int_{(\P^2)^{[4]}} \left(4H+E\right)^8 = \frac{1692165}{8!}.
\end{align*}
This self intersection number was computed with the equivariant localization formula. With a computer one can also calculate the Euclidean volume of the upper bound,
\[ \vol_{\R^8}\left(\overline{\Delta}\left(4H+E\right)\right) = \frac{112811}{2688} = \frac{1692165}{8!}. \]
By Theorem \ref{projub} we have the inclusion $\Delta\left(4H+E\right) \subseteq \overline{\Delta}\left(4H+E\right)$, but these two convex bodies have the same volumes so they must be equal.\\

Surprisingly, this one calculation implies that $\overline{\Delta}(D) = \Delta(D)$ for all remaining divisors as well. Indeed, we first consider divisors of the form $(4-t)H+E$ for some real number $t\geq 0$. One way to handle these divisors is to apply Theorem 4.24 from \cite{LM} on slices of Newton-Okounkov bodies. With our usual coordinates $(a_1,\dots,a_4,b_1,\dots,b_4)$ on $\R^8$, the theorem implies that for any divisor $D\in \Pic(X^{[n]})$ and $t\geq 0$ we have
\[ \Delta\left(D-tH\right) = \Delta\left(D\right)_{a_1\geq t} - (t,0,\dots,0.) \]
In other words, the Newton-Okounkov body $\Delta\left(D-tH\right)$ is equal to the part of the Newton-Okounkov body $\Delta\left(D\right)$ with first coordinate at least $t$, shifted down by $t$ in the first coordinate. There is a subtlety in that our valuation $\nu$ is not defined using flags as in \cite{LM}, so the theorem does not strictly apply as stated. However, the first coordinate of $\nu$ is equal to the order of vanishing of the corresponding section along a divisor of class $H$, so the result still holds. One can also check this property directly in this case, at the level of polynomials and trailing terms. It follows from the defining inequalities that $\overline\Delta(D-tH) = \overline\Delta(D)_{a_1\geq t}-(t,0,\dots,0)$ as well for all $D\in \Pic((\P^2)^{[4]})$. This, with the homogeneity of Newton-Okounkov bodies, implies that $\overline{\Delta}(D)= \Delta(D)$ for all divisors $D$ in the cone spanned by $4H+E$ and $3H+2E$.\\

This leaves the divisors in the cone spanned by $4H+E$ and $H$. We have $\overline\Delta(D)=\Delta(D)$ for $D$ on the boundary rays of this cone, and $\overline\Delta(D)$ varies linearly on the cone. But Newton-Okounkov bodies are super-additive, in the sense that 
\[ \Delta(xD+yD') \supseteq x\Delta(D)+y\Delta(D') \]
for any divisors $D,D'$ and $x,y\geq 0$. We conclude that $\overline{\Delta}(D) = \Delta(D)$ for divisors in this final region as well.
\qed
\end{example}

The decomposition of $\Eff((\P^2)^{[4]})$ into chambers on which the Newton-Okounkov bodies vary linearly corresponds to a Minkowski basis for $\Delta((\P^2)^{[4]})$ in the terminology of \cite{SS}. In this case the Minkowski basis decomposition coincides with the stable base locus decomposition of $(\P^2)^{[4]}$ \cite{ABCH}, but these decompositions appear to differ for $n>4$.\\

One can show that the convex bodies $\overline{\Delta}(D)$ vary linearly on the nef cone of $(\P^2)^{[n]}$ for any $n$, so the argument given above can be applied to any of the Hilbert schemes $(\P^2)^{[n]}$: Pick any ample divisor $D\in \Pic((\P^2)^{[n]})$ (the nef divisor $D = (n-1)H+E$ would also work), and compute both the Euclidean volume $\vol_{\R^{2n}}\overline{\Delta}(D)$, and the top self-intersection number $\int_{(\P^2)^{[n]}}D^{2n}$. If these numbers agree up to the factor of $(2n)!$, then Conjecture \ref{sharp} holds for all divisors on $(\P^2)^{[n]}$ for the given $n$. \\

The top self-intersection numbers can be computed quickly, even for relatively large $n$, using the equivariant localization formula. It is much more difficult to compute the volumes of the convex bodies $\overline{\Delta}(D)$. In the case $n=4$ above, the convex body $\Delta(4H+E)$ whose volume we computed is a polytope in $\R^8$ with $186$ vertices. For $n=5$, the convex set $\overline\Delta(4H+E)$ is a polytope in $\R^{10}$ with $581$ vertices. The complexity of these polytopes makes it impractical for verify Conjecture \ref{sharp} for large $n$ in this way.\\

We have similarly checked Conjecture \ref{sharp} for nef divisors on the Hilbert schemes of small numbers of points on $\P^1\times \P^1$, and the first several Hirzebruch surfaces. One can show that when the surface $X$ is $\P^1\times\P^1$ or a Hirzebruch surface, the polytopes $\overline{\Delta}(D)$ again vary linearly for $D$ in the nef cone of $X^{[n]}$. The increased Picard rank of these surfaces, however, means that the theorem on slices of Newton-Okounkov bodies (Theorem 4.24 of \cite{LM}) fails to cover the remaining effective divisors. Our justification for asserting Conjecture \ref{sharp} for non-nef divisors as well comes from the data computed in the final section about the cones of effective divisors.

\subsection{The Cone of Effective Divisors on $X^{[n]}$}\label{effsection}

Characterizing the effective divisors on $X^{[n]}$ appears to be a subtle problem (see \cite{BC}, Section 3). Huizenga has computed the effective cones on $(\P^2)^{[n]}$ for all $n$, which depend on the slopes of stable vector bundles on $\P^2$ \cite{Hu}. For other surfaces, the effective cones are known only for small values of $n$.\\

In Proposition \ref{projsections} we identified the global sections of $\Oo(D_n+rE)\in \Pic(X^{[n]})$ with the set $ A(D_n+rE)$, consisting of the (anti-)symmetric polynomials contained in the ideal $J^r$ that satisfy a term condition determined by $D$. However, it is unclear from the definition of $A(D_n+rE)$ even when these spaces are zero, i.e. when $D_n+rE$ is effective. If one knew the Newton-Okounkov bodies exactly, then one would also know the set of effective divisors, since $\Delta(D_n+rE)$ is nonempty precisely when $D_n+rE$ is pseudo-effective. We have an upper bound for the Newton-Okounkov bodies, so we obtain a corresponding upper bound for the effective cones. 

\begin{corollary}\label{eff}
For any effective divisor $D_n+rE$ on $X^{[n]}$, the convex set $\overline{\Delta}(D_n+rE)\subseteq \R^{2n}$ is nonempty. 
\end{corollary}

\begin{proof}
Since $D_n-\frac r2 B$ is effective, the Newton-Okounkov body $\Delta(D_n-\frac{r}{2}B)$ is nonempty and is contained in $\overline{\Delta}(D_n+rE)$ by Theorem \ref{projub}.
\end{proof}

This corollary can be used to show that divisors are not effective, and therefore implies an upper bound for the effective cone of $X^{[n]}$. This upper bound is best understood via the global Newton-Okounkov body, as we now explain. \\

One can define the Newton polygon of any class $D\in N^1(X)_\R$. Similarly, we extend the sets $\overline{\Delta}(D_n+rE)$ to all real classes $D_n+rE\in N^1(X^{[n]})_\R$ using the same inequalities given in Theorem \ref{projub}. Let $\overline{\Delta}(X^{[n]})\subseteq N^1(X^{[n]})_\R\times \R^{2n}$ be the set whose fiber over any real class $\xi\in N^1(X^{[n]})_\R$ is $\overline{\Delta}(\xi)$. It follows from the defining equations given in Theorem \ref{projub} and the convexity of global Newton-Okounkov body of $X$ that $\overline{\Delta}(X^{[n]})$ is a closed, convex, polyhedral cone.\\

This set $\overline{\Delta}(X^{[n]})$ is an upper bound for the global Newton-Okounkov body $\Delta(X^{[n]})$, whose fiber over a divisor is its exact Newton-Okounkov body. The global Newton-Okounkov body projects precisely onto the effective cone $\Delta(X^{[n]})\to N^1(X^{[n]})_\R$. The previous corollary can therefore be rephrased as follows.

\begin{corollary}\label{eff2}
The image of the projection $\overline{\Delta}(X^{[n]})\to N^1(X^{[n]})_\R$ contains the cone of effective divisors. \qed
\end{corollary}

The advantage of this phrasing is that computing the linear images of these polyhedra can be reduced to linear optimization problems.\\

Conjecture \ref{sharp} would imply that the upper bound of Corollary \ref{eff2} equals the exact effective cone for Hilbert schemes of points on $\P^2$, $\P^1\times\P^1$, and Hirzebruch surfaces. We have verified that this holds for $n\leq 171$ points on $\P^2$ \textit{numerically}. For example, for the Hilbert scheme of $32$ points on $\P^2$, we set up the linear optimization problem to give a lower bound on $\mu$, where $\mu H+E$ lies on the boundary of the effective cone of $(\P^2)^{[32]}$. We numerically approximated the solution to be 6.57894736842105. The exact solution to the optimization problem is easily seen to be rational, and the numerical approximation is within $10^{-15}$ (all of the digits shown) of the exact value of $\mu$, which is 125/19.\\

The table on the final page contains upper bounds for certain slices of effective cones computed \textit{numerically} for the Hilbert schemes of points on $\P^2,\P^1\times\P^1$, and the Hirzebruch surfaces $\mathscr H_1$, and $\mathscr H_2$. These numbers were obtained using the same shortcut of numerical approximation, followed by finding an unusually close, simple rational number.\\

For the surfaces other than $\P^2$, these values do not determine the entire effective cone due to the larger Picard rank, but it is possible to use Corollary \ref{eff2} to compute the entire effective cone bounds in these cases as well. Ryan \cite{Ry} has computed the effective cones on $\P^1\times \P^1$ for $n\leq 16$, which coincide with the upper bound of Corollary \ref{eff2} in each case. For $n\geq 17$ points on $\P^1\times\P^1$, there is an upper bound on the effective cone of $(\P^1\times\P^1)^{[n]}$ obtained in \cite{BC} (Example 3.9) coming from certain moving curve classes. We have checked \textit{numerically} for all $n\leq 100$ that the effective cone bound from Corollary \ref{eff2} satisfies the inequalities obtained in \cite{BC}.\\

In many cases, the upper bound of Corollary \ref{eff2} implies new inequalities bounding the effective cones. For example, we have found that for $n=17$ points on $\P^1\times\P^1$, any effective divisor $xH_1+yH_2+zE$ satisfies $8x+5y\geq 40z$ (and symmetrically $5x+8y\geq 40z$). Here we use $H_1$ and $H_2$ to denote the images of the classes of lines $\{p\}\times \P^1$ and $\P^1\times\{p\}$ in $\Pic((\P^1\times \P^1)^{[n]}) $. More generally we have observed, but not proved, that for any $k\geq 0$ and $n=17+6k$ there appear to be conditions $5x+8y\geq (40+16k)z$ and $8x+5y\geq (40+16k)z$ restricting effective divisors $xH_1+yH_2+zE$. We have observed many other similar families of inequalities on the effective cones of $\P^1\times\P^1$, $\mathscr H_1$, and $\mathscr H_2$. It would be interesting to find families of moving curve classes corresponding to these inequalities. \\

\newpage

\begin{minipage}{0.25\columnwidth}

\centering
\vspace{-2.3cm}

\begin{tikzpicture}
    \draw (0,0)-- (1,0) -- (0,1) -- (0,0);
    \filldraw (0,0) circle (2pt);
    \filldraw (0,1) circle (2pt);
    \filldraw (1,0) circle (2pt);
\end{tikzpicture}

\vspace{1cm}

$\P^2$ \\

\vspace{1ex}

\begin{tabular}{ c c }
$n$ & $\mu$ \\\hline
2& 1\\
3& 1\\
4& 3/2\\
5& 2\\
6& 2\\
7& 12/5\\
8& 8/3\\
9& 3\\
10& 3\\
11& 10/3\\
12& 7/2\\
13& 15/4\\
14& 4\\
15& 4\\
16& 30/7\\
17& 40/9\\
18& 23/5\\
19& 24/5\\
20& 5\\
21& 5\\
22& 21/4\\
23& 43/8\\
24& 11/2\\
25& 17/3\\
26& 35/6\\
27& 6\\
28& 6\\
29& 56/9\\
30& 19/3\\
31& 84/13\\
32& 125/19\\
33& 47/7\\
34& 48/7\\
35& 7\\
36& 7\\
37& 36/5\\
38& 73/10\\
39& 37/5\\
40& 15/2\\
\end{tabular}
\end{minipage}
\begin{minipage}{0.25\columnwidth}

\centering

\vspace{-2.3cm}

\begin{tikzpicture}
    \draw (0,0) -- (1,0) -- (1,1) -- (0,1) -- (0,0);
    \filldraw (0,0) circle (2pt);
    \filldraw (0,1) circle (2pt);
    \filldraw (1,0) circle (2pt);
    \filldraw (1,1) circle (2pt);
\end{tikzpicture}

\vspace{1cm}

$\P^1\times\P^1$

\vspace{1ex}

\begin{tabular}{ c c }
$n$ & $\mu$ \\\hline
2 & 1/2 \\
3& 1\\
4& 1\\
5& 4/3\\
6& 3/2\\
7& 7/4\\
8& 2\\
9& 2\\
10& 9/4\\
11& 12/5\\
12& 5/2\\
13& 8/3\\
14& 17/6\\
15& 3\\
16& 3\\
17& 16/5\\
18& 33/10\\
19& 24/7\\
20& 7/2\\
21& 40/11\\
22& 15/4\\
23& 31/8\\
24& 4\\
25& 4\\
26& 25/6\\
27& 17/4\\
28& 13/3\\
29& 40/9\\
30& 9/2\\
31& 60/13\\
32& 47/10\\
33& 24/5\\
34& 49/10\\
35& 5\\
36& 5\\
37& 36/7\\
38& 73/14\\
39& 37/7\\
40& 59/11\\
\end{tabular}
\end{minipage}
\begin{minipage}{0.25\columnwidth}

\centering

\vspace{-2.8cm}

\begin{tikzpicture}
    \draw (0,0) -- (1,0) -- (1,2) -- (0,1) -- (0,0);
    \filldraw (0,0) circle (2pt);
    \filldraw (0,1) circle (2pt);
    \filldraw (1,0) circle (2pt);
    \filldraw (1,1) circle (2pt);
    \filldraw (1,2) circle (2pt);
\end{tikzpicture}

\vspace{0.5cm}

$\mathscr H_1$

\vspace{1ex}

\begin{tabular}{ c c }
$n$ & $\mu$ \\\hline
2& 1/2\\
3& 2/3\\
4& 1\\
5& 1\\
6& 5/4\\
7& 7/5\\
8& 8/5\\
9& 5/3\\
10& 11/6\\
11& 2\\
12& 2\\
13& 24/11\\
14& 16/7\\
15& 19/8\\
16& 5/2\\
17& 21/8\\
18& 8/3\\
19& 25/9\\
20& 26/9\\
21& 3\\
22& 3\\
23& 22/7\\
24& 45/14\\
25& 33/10\\
26& 101/30\\
27& 52/15\\
28& 39/11\\
29& 40/11\\
30& 11/3\\
31& 15/4\\
32& 23/6\\
33& 47/12\\
34& 4\\
35& 4\\
36& 70/17\\
37& 71/17\\
38& 161/38\\
39& 56/13\\
40& 157/36\\
\end{tabular}
\end{minipage}
\begin{minipage}{0.25\columnwidth}

\centering

\vspace{-3.3cm}

\begin{tikzpicture}
    \draw (0,0) -- (1,0) -- (1,3) -- (0,1) -- (0,0);
    \filldraw (0,0) circle (2pt);
    \filldraw (0,1) circle (2pt);
    \filldraw (1,0) circle (2pt);
    \filldraw (1,1) circle (2pt);
    \filldraw (1,2) circle (2pt);
    \filldraw (1,3) circle (2pt);
\end{tikzpicture}

$\mathscr H_2$\\

\vspace{1ex}

\begin{tabular}{ c c }
$n$ & $\mu$ \\\hline
2& 1/3\\
3& 2/3\\
4& 3/4\\
5& 1\\
6& 1\\
7& 6/5\\
8& 4/3\\
9& 16/11\\
10& 11/7\\
11& 12/7\\
12& 7/4\\
13& 15/8\\
14& 2\\
15& 2\\
16& 15/7\\
17& 20/9\\
18& 23/10\\
19& 12/5\\
20& 37/15\\
21& 28/11\\
22& 29/11\\
23& 30/11\\
24& 11/4\\
25& 17/6\\
26& 35/12\\
27& 3\\
28& 3\\
29& 28/9\\
30& 19/6\\
31& 42/13\\
32& 23/7\\
33& 64/19\\
34& 24/7\\
35& 73/21\\
36& 53/15\\
37& 18/5\\
38& 11/3\\
39& 56/15\\
40& 15/4\\

\end{tabular}
\end{minipage}

\vfill

Each polygon $P_D$ corresponds to a divisor $D$ on a specified toric surface $X$. For each $n\geq 2$, any effective divisor of the form $t D_n+E$ on $X^{[n]}$ has $t\geq \mu$. These $\mu$'s are conjectured to be optimal, i.e. $ \mu D_n+E$ conjecturally lies on the boundary of the effective cone of $X^{[n]}$. 

\pagebreak


\begin{thebibliography}{9}
\bibitem{ABCH} 
D. Arcara, A. Bertram, I. Coskun, J. Huizenga,
\textit{The minimal model program for the Hilbert scheme of points on P2 and Bridgeland stability}. 
Adv. Math, Vol. 235 (2013), p. 580-626.

\bibitem{BC}
A. Bertram, I. Coskun
\textit{The Birational Geometry of the Hilbert Scheme of Points on Surfaces}.
Birational Geometry, Rational Curves, and Arithmetic (2013), p. 15-55.

\bibitem{B}
S. Boucksom
\textit{Corps d'Okounkov [d'apr\`es Okounkov, Lazarsfeld-Musta\c{t}\u{a}, et Kaveh-Khovanskii]}.
S\'eminaire Bourbaki 65\`eme ann\'ee, (2012-2013), no. 1059, p. 1-41.

\bibitem{BP}
M. Brion, C. Procesi
\textit{Action d’un tore dans une vari\'et\'e projective}.
Operator algebras, Unitary Representations, Enveloping Algebras, and Invariant Theory, Progress in Mathematics Vol. 92, (1990) Birkh\"auser, p. 509–539.

\bibitem{EGL}
\textit{On the cobordism class of the Hilbert scheme of a surface}.
J. Algebraic Geom, Vol. 10, (2001) no. 1, p. 81–100.

\bibitem{Fo1}
J. Fogarty,
\textit{Algebraic families on an algebraic surface}.
Am. J. Math, Vol. 90 (1968), p. 511-521.

\bibitem{Fo2}
J. Fogarty,
\textit{Algebraic families on an algebraic surface II: the Picard scheme of the punctual Hilbert scheme}.
Am. J. Math, Vol. 95 (1973), p. 660-687.

\bibitem{Fu}
W. Fulton,
\textit{Introduction to Toric Varieties}.
Annals of Math. Studies, vol. 131, Princeton Univ. Press, Princeton, 1993.

\bibitem{Ha1}
M. Haiman,
\textit{Notes on Macdonald polynomials and the geometry of Hilbert schemes}.
(https://math.berkeley.edu/$\sim$ mhaiman/)

\bibitem{Ha2}
M. Haiman,
\textit{t, q-Catalan numbers and the Hilbert scheme}.
Discrete Math, Vol 193 (1998) no. 1, p. 201-224.

\bibitem{Hart}
R. Hartshorne
\textit{Algebraic Geometry}.
Graduate Texts in Mathematics, No. 52, Springer-Verlag (1977)

\bibitem{Hu}
J. Huizenga,
\textit{Effective divisors on the Hilbert scheme of points in the plane and interpolation for stable bundles}.
J. Algebraic Geom., Vol 25 (2016) no. 1, p. 19-75.

\bibitem{KK}
K. Kaveh, A. Khovanskii,
\textit{Newton-Okounkov bodies, semigroups of integral points, graded
algebras and intersection theory}.
Ann. of Math. (2) 176 (2012), no. 2, p. 925-978.

\bibitem{KK2}
K. Kaveh, A. Khovanskii,
\textit{Convex bodies associated to actions of reductive groups}.
AMosc. Math. J., Vol 12 (2012), no. 2, p. 369-396.

\bibitem{LM} 
R. Lazarsfeld, M. Musta\c{t}\u{a},
\textit{Convex bodies associated to linear series}. 
Annales scientifiques de l'École Normale Supérieure,  Serie 4,  Volume 42 (2009) no. 5,  p. 783-835.

\bibitem{O1}
A. Okounkov,
\textit{Brunn-Minkowski inequality for multiplicities}.
Invent. Math 125 (1996), no. 3, p. 405-411.

\bibitem{O2}
A. Okounkov,
\textit{Why would multiplicities be log-concave?}
Progr. Math. 213, (2003), p. 329-347.

\bibitem{Ry}
T. Ryan, 
\textit{The effective cone of moduli spaces of sheaves on a smooth quadric surface}.
Nagoya Mathematical Journal, 232, (2018), p. 151-215.

\bibitem{SS}
W. Sawin, D. Schmitz,
\textit{On numerical Newton-Okounkov bodies and the existence of Minkowski bases}.
arXiv:1607.03667 [math.AG]

\end{thebibliography}
\end{document}